\documentclass[preprint]{elsarticle}
\usepackage{amsmath,amssymb}
\usepackage{booktabs}
\usepackage{dsfont}
\usepackage{cases}
\usepackage{natbib}
\usepackage{rotating}
\renewcommand{\leq}{\leqslant}
\renewcommand{\geq}{\geqslant}
\newcommand{\bp}{\boldsymbol{p}}
\renewcommand{\phi}{\varphi}
\renewcommand{\a}{\alpha}

\newcommand{\id}{\mathrm{id}}

\renewcommand{\P}{\mathbb{P}}
\newcommand{\E}{\mathbb{E}}

\newcommand{\R}{\mathbb{R}}

\newcommand{\cX}{\mathcal{X}}
\newcommand{\cF}{\mathcal{F}}

\newcommand{\cP}{\mathcal{P}}
\newcommand{\cA}{\mathcal{A}}

\newcommand{\cC}{\mathcal{C}}

\def\ds1{\mathds{1}}
\renewcommand{\epsilon}{\varepsilon}

\newcommand{\wh}{\widehat}
\newcommand{\argmax}{\mathop{\mathrm{argmax}}}

\renewcommand{\mod}[2]{[#1 \,\, \mathrm{mod} \,\, #2]}

\newlength{\minipagewidth}
\newcommand{\bookbox}[1]{
\par\medskip\noindent
\framebox[\textwidth]{
\begin{minipage}{\minipagewidth}
{#1}
\end{minipage} } \par\medskip }

\newtheorem{theorem}{Theorem}
\newtheorem{corollary}{Corollary}
\newdefinition{remark}{Remark}
\newtheorem{lemma}{Lemma}
\newtheorem{proposition}{Proposition}
\newproof{proof}{Proof}
\newtheorem{example}{Example}
\newenvironment{ex}{\begin{example} \rm}{\end{example}}

\begin{document}

\begin{frontmatter}

\title{Pure Exploration in Finitely--Armed and Continuous--Armed Bandits}

\author{S\'ebastien Bubeck\corref{cor}}
\address{ INRIA Lille -- Nord Europe, SequeL project, \\
40 avenue Halley, 59650 Villeneuve d'Ascq, France}
\ead{sebastien.bubeck@inria.fr}

\author{R\'emi Munos}
\address{INRIA Lille -- Nord Europe, SequeL project, \\
40 avenue Halley, 59650 Villeneuve d'Ascq, France}
\ead{remi.munos@inria.fr}

\author{Gilles Stoltz}
\address{Ecole Normale Sup{\'e}rieure, CNRS \\
75005 Paris, France \\
\& \\
HEC Paris, CNRS, \\
78351 Jouy-en-Josas, France \\}
\ead{gilles.stoltz@ens.fr}
\cortext[cor]{Corresponding author.}

\begin{abstract}
We consider the framework of stochastic multi-armed bandit problems and study the possibilities
and limitations of forecasters that perform an on-line exploration of the arms.
These forecasters are assessed in terms of their simple regret, a regret
notion that captures the fact that exploration is only constrained by the number
of available rounds (not necessarily known in advance), in contrast to the case when the cumulative
regret is considered and when exploitation needs to be performed at the same time. We
believe that this performance criterion is suited to situations when
the cost of pulling an arm is expressed in terms of resources rather than rewards.
We discuss the links between the simple and the cumulative regret.
One of the main results in the case of a finite number of arms
is a general lower bound on the simple regret of a forecaster in terms of its
cumulative regret: the smaller the latter, the larger the former. Keeping this result in mind,
we then exhibit upper bounds on the simple regret of some forecasters.
The paper ends with a study devoted to continuous-armed bandit problems;
we show that the simple regret can be minimized with respect to a family of probability
distributions if and only if the cumulative regret
can be minimized for it. Based on this equivalence, we are able to prove that the separable metric spaces
are exactly the metric spaces on which
these regrets can be minimized
with respect to the family of all probability distributions with continuous mean-payoff functions.
\end{abstract}

\end{frontmatter}

\section{Introduction}

Learning processes usually face an exploration versus exploitation
dilemma, since they have to get information on the environment (exploration)
to be able to take good actions (exploitation). A key example is
the multi-armed bandit problem \cite{Rob52}, a sequential decision
problem where, at each stage, the forecaster has to pull one out of
$K$ given stochastic arms and gets a reward drawn at random
according to the distribution of the chosen arm.
The usual assessment criterion of a forecaster is given by its cumulative regret,
the sum of differences between the expected reward of the best arm and the obtained rewards.
Typical good forecasters, like UCB \cite{ACF02}, trade off between exploration and exploitation.

Our setting is as follows. The forecaster may sample the arms a given number of times $n$ (not necessarily known
in advance) and is then asked to output a recommended arm. He is evaluated by his simple regret, that is, the difference between the
average payoff of the best arm and the average payoff obtained by his recommendation.
The distinguishing feature from the classical multi-armed bandit problem is that
the exploration phase and the evaluation phase are separated.
We now illustrate why this is a natural framework for numerous applications.

Historically, the first occurrence of multi-armed bandit problems was given by medical trials.
In the case of a severe disease, ill patients only are included in the trial
and the cost of picking the wrong treatment
is high (the associated reward would equal a large negative value).
It is important to minimize the cumulative regret, since the test
and cure phases coincide. However, for cosmetic products, there exists a test phase separated from
the commercialization phase, and one aims at minimizing the regret of the commercialized product rather
than the cumulative regret in the test phase, which is irrelevant.
(Here, several formul{\ae} for a cream are considered and some quantitative measurement, like
skin moisturization, is performed.)

\medskip
The pure exploration problem addresses the design of strategies making the best possible use of available numerical
resources (e.g., as {\sc cpu} time) in order to optimize the performance of some decision-making task.
That is, it occurs in situations with a preliminary exploration phase
in which costs are not measured in terms of rewards but rather in terms of resources, that come in limited budget.

A motivating example concerns recent works on computer-go (e.g., the MoGo program \cite{GWMT06}).
A given time, i.e., a given amount of {\sc cpu} times is given to the player
to explore the possible outcome of sequences of plays and output a final decision.
An efficient exploration of the search space is obtained by considering a hierarchy of
forecasters minimizing some cumulative regret~-- see, for instance,
the {\sc uct} strategy \cite{KS06} and the {\sc bast} strategy \cite{CM07}.
However, the cumulative regret does not seem to be the right
way to base the strategies on, since the simulation costs are the same for exploring all
options, bad and good ones.
This observation
was actually the starting point of the notion of simple regret and of this work.

A final related example is the maximization of some function $f$, observed with noise, see, e.g., \cite{Kle04,BMSS09}.
Whenever evaluating $f$ at a point is costly (e.g., in terms of numerical or financial costs),
the issue is to choose as adequately as possible where to query the value of this function
in order to have a good approximation to the maximum.
The pure exploration problem considered here addresses exactly the design of adaptive exploration strategies
making the best use of available resources in order to make the most precise prediction
once all resources are consumed.

As a remark, it also turns out that in all examples considered above,
we may impose the further restriction that the forecaster ignores ahead of time
the amount of available resources (time, budget, or the number of patients to be included)~--
that is, we seek for anytime performance.

\medskip
The problem of pure exploration presented above was referred to as ``budgeted multi-armed bandit problem'' in
the open problem \cite{MLG04} (where, however, another notion of regret than simple regret is considered).
The pure exploration problem was solved in a minmax sense for the case of two arms
only and rewards given by probability distributions over $[0,1]$ in \cite{Sch06}.
A related setting is considered in \cite{EMM02} and \cite{MT04}, where
forecasters perform exploration during a random number of
rounds $T$ and aim at identifying an $\epsilon$--best arm. These articles study the possibilities
and limitations of policies achieving this goal with overwhelming $1-\delta$ probability
and indicate in particular upper and lower bounds on (the expectation of) $T$.
Another related problem is the identification of the best arm (with high probability).
However, this binary assessment criterion (the forecaster is either right or wrong
in recommending an arm) does not capture the possible closeness in performance of the recommended arm compared to the optimal one,
which the simple regret does. Moreover unlike the latter, this criterion is not suited for a distribution-free analysis.

\subsection*{Contents and structure of the paper}

We present formally the model in Section~\ref{sec:setup} and
indicate therein that our aim is to study the links between the simple and the cumulative regret. Intuitively, an
efficient allocation strategy for the simple regret should rely on some exploration--exploitation trade-off
but the rest of the paper shows that this trade-off is not exactly the same as in the case of the cumulative regret.

Our first main contribution (Theorem~\ref{th:MainDistDep}, Section~\ref{sec:smallbad})
is a lower bound on the simple regret in terms of the cumulative regret suffered in the exploration phase,
which shows that the minimal simple regret is larger as the bound on the cumulative regret is smaller.
This in particular implies that the uniform exploration of the arms is a good benchmark when the number
of exploration rounds $n$ is large.

In Section~\ref{sec:UBS} we then study the simple regret of some natural forecasters, including the one
based on uniform exploration, whose simple regret vanished exponentially fast.
(\emph{Note}: The upper bounds presented in this paper can however be improved by
the recent results of~\cite{AuBuMu10}.)
In Section~\ref{sec:comparison}, we show how one can somewhat circumvent the fundamental lower bound indicated above:
some strategies designed to have a small cumulative regret can outperform (for small or moderate values of $n$)
strategies with exponential rates of convergence for their simple regret; this is shown both by means of a
theoretical study and by simulations.

Finally we investigate in Section~\ref{sec:topo} the continuous-armed bandit
problem where the set of arms is a topological space. In this setting we use the simple
regret as a tool to prove that the separable metric spaces are exactly
the metric spaces for which it is possible to have a sublinear cumulative regret
with respect to the family of all probability distributions with continuous mean-payoff functions.
This would be our second main contribution.

\section{Problem setup, notation, structure of the paper}
\label{sec:setup}

We consider a sequential decision problem given by stochastic multi-armed bandits.
A finite number $K \geq 2$ of arms, denoted by $i = 1,\ldots,K$, are available and the $i$--th of them is
parameterized by a fixed (unknown) probability distribution $\nu_i$ over $[0,1]$, with expectation
denoted by $\mu_i$. At those rounds when it is pulled,
its associated reward is drawn at random according to $\nu_i$, independently of all previous rewards.
For each arm $i$ and all time rounds $n \geq 1$, we denote by $T_{i}(n)$ the number of times arm $i$ was pulled
from rounds 1 to $n$, and by $X_{i,1}, X_{i,2}, \ldots, X_{i,T_{i}(n)}$ the sequence of associated rewards.

\medskip
The forecaster has to deal simultaneously with two tasks,
a primary one and a secondary one.
The secondary task consists in exploration, i.e., the forecaster
should indicate at each round $t$ the arm $I_t$ to
be pulled, based on past rewards (so that $I_t$ is a random variable).
Then the forecaster gets to see the associated reward $Y_t$, also denoted by $X_{I_t,T_{I_t}(t)}$ with the notation above.
The sequence of random variables $(I_t)$ is referred to as an allocation strategy.
The primary task is to output at the end of each round $t$ a recommendation
$J_t$ to be used in a one-shot instance if/when the environment sends some stopping signal meaning that the
exploration phase is over. The sequence
of random variables $(J_t)$ is referred to as a recommendation strategy.
In total, a forecaster is given by an allocation and a recommendation strategy.

Figure~\ref{fig:descr} summarizes the description of the sequential game
and points out that the information available to the forecaster for choosing
$I_t$, respectively $J_t$, is formed by the $X_{i,s}$ for $i = 1,\ldots,K$ and
$s = 1,\ldots,T_i(t-1)$, respectively, $s = 1,\ldots,T_i(t)$. Note that we also allow the forecaster
to use an external randomization in the definition of $I_t$ and $J_t$.
\begin{figure}[t]
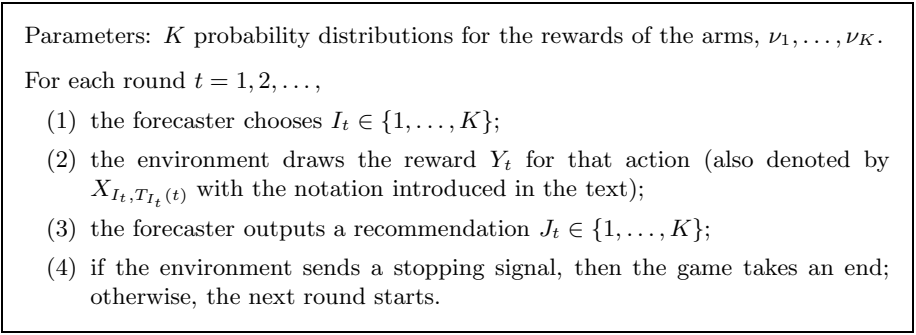

\bookbox{\small
Parameters: $K$ probability distributions for the rewards
of the arms, $\nu_1,\ldots,\nu_K$.

\medskip\noindent
For each round $t=1,2,\ldots,$
\begin{itemize}
\item[(1)]
the forecaster chooses $I_t \in \{1,\hdots,K\}$;
\item[(2)]
the environment draws the reward $Y_t$ for that action (also denoted by $X_{I_t, T_{I_t}(t)}$
with the notation introduced in the text);
\item[(3)]
the forecaster outputs a recommendation $J_t \in \{ 1,\ldots,K \}$;
\item[(4)]
if the environment sends a stopping signal, then the game takes an end;
otherwise, the next round starts.
\end{itemize}
}
\caption{\label{fig:descr}
The pure exploration problem for multi-armed bandits (with a finite number of arms).}
\end{figure}

\medskip
As we are only interested in the performances of the recommendation strategy $(J_t)$,
we call this problem the pure exploration problem for multi-armed bandits
and evaluate the forecaster through its simple regret, defined as follows.
First, we denote by
\[
\mu^* = \mu_{i^*} = \max_{i=1,\ldots,K} \mu_i
\]
the expectation of the rewards of
the best arm $i^*$ (a best arm, if there are several of them
with same maximal expectation).
A useful notation in the sequel is the gap $\Delta_i = \mu^* - \mu_i$
between the maximal expected reward and the one of the $i$--th arm;
as well as the minimal gap
\[
\Delta = \min_{i : \Delta_i > 0} \Delta_i~.
\]
Now, the simple regret at round $n$ equals the regret on a one-shot instance of the game for the recommended arm $J_n$,
that is, put more formally,
\[
r_n = \mu^* - \mu_{J_n} = \Delta_{J_n}~.
\]
A quantity of related interest is the cumulative regret at round $n$,
which is defined as
\[
R_n = \sum_{t=1}^n \mu^* - \mu_{I_t}~.
\]
A popular treatment of the multi-armed bandit problems
is to construct forecasters ensuring that $\E R_n = o(n)$,
see, e.g., \cite{LR85} or \cite{ACF02}, and even $R_n = o(n)$ a.s., as follows,
e.g., from \cite[Theorem 6.3]{ACFS02} together with the Borel--Cantelli lemma.
The quantity $r'_t = \mu^* - \mu_{I_t}$ is sometimes called instantaneous
regret. It differs from the simple regret $r_t$ and in particular,
$R_n = r'_1 + \ldots + r'_n$ is in
general not equal to $r_1 + \ldots + r_n$.
Theorem~\ref{th:MainDistDep}, among others,
will however indicate some connections between $r_n$ and $R_n$.

\begin{remark}
The setting described above is concerned with a finite number of arms.
In Section~\ref{sec:topo} we will extend it to the case of arms indexed by
a general topological space.
\end{remark}

\section{The smaller the cumulative regret, the larger the simple regret}
\label{sec:smallbad}

It is immediate that for well-chosen recommendation strategies, the simple regret can be upper bounded in
terms of the cumulative regret. For instance, the strategy that at time $n$ recommends arm $i$ with probability $T_i(n)/n$
(recall that we allow the forecaster to use an external randomization) ensures that
the simple regret satisfies $\E r_n = \E R_n/n$. Therefore, upper bounds on $\E R_n$ lead to upper bounds
on $\E r_n$.

We show here that, conversely,
upper bounds on $\E R_n$ also lead to lower bounds
on $\E r_n$: the smaller the guaranteed upper bound on $\E R_n$,
the larger the lower bound on $\E r_n$, no matter what the
recommendation strategy is.

\medskip
This is interpreted as a variation of the ``classical'' trade-off between exploration and exploitation.
Here, while the recommendation strategy $(J_n)$ relies only on the exploitation of the results
of the preliminary exploration phase, the design of the allocation strategy $(I_t)$ consists in an efficient exploration
of the arms. To guarantee this efficient exploration, past payoffs of the arms have to be considered and thus, even in the
exploration phase, some exploitation is needed.
Theorem~\ref{th:MainDistDep} and its corollaries aim at quantifying the needed respective
amount of exploration and exploitation.
In particular,
to have an asymptotic optimal rate of decrease for the simple regret,
each arm should be sampled a linear number of times, while
for the cumulative regret, it is known that
the forecaster should not do so more than a logarithmic number of times on the suboptimal arms.

Formally, our main result is reported below in Theorem~\ref{th:MainDistDep}.
It is strong in the sense that it lower bounds the simple regret of any forecaster for all
possible sets of Bernoulli distributions $\{ \nu_1,\ldots,\nu_K \}$ over the rewards with parameters
that are all distinct (no two parameters can be equal) and all different from $1$. Note however
that in particular these conditions entail that there is a unique best arm.

\begin{theorem}[Main result]
\label{th:MainDistDep}
For any forecaster (i.e., for any pair of allocation and recommendation strategies)
and any function $\epsilon : \{ 1,2,\ldots\} \to \R$ such that
\begin{quote}
{\small
for all (Bernoulli) distributions $\nu_1,\ldots,\nu_K$ on the rewards, there exists a constant
$C \geq 0$ with $\E R_n \leq C \, \epsilon(n)$, }
\end{quote}
the following holds true:
\begin{quote}
{\small
for all sets of $K \geq 3$ Bernoulli distributions on the rewards, with parameters that are all distinct and all different from $1$,
there exists a constant
$D \geq 0$ and an ordering $\nu_1,\ldots,\nu_K$ of the considered distributions such that
\[
\E r_n \geq \frac{\Delta}{2} \, e^{- D \epsilon(n)}~.
\]
}
\end{quote}
\end{theorem}

We insist on the fact that only \emph{sets}, that is, unordered collections, of distributions
are considered in the second part of the statement of the theorem.
Put differently, we merely show therein that for each ordered $K$--tuple of distributions that are as indicated above, there exists a
reordering that leads to the stated lower bound on the simple regret.
This is the best result that can be achieved. Indeed, some forecasters are sensitive to the ordering of the distributions
and might get a zero regret for a significant fraction of the
ordered $K$--tuples simply because, e.g., their strategy is to constantly pull a given arm,
which is sometimes the optimal strategy just by chance.
To get lower bounds in all cases we must therefore allow reorderings of $K$--tuples
(or, equivalently, orderings of sets).

\begin{corollary}[General distribution-dependent lower bound]
\label{cor:LBunif}
For any {fo\-re\-cas\-ter}, and
any set of $K \geq 3$ Bernoulli distributions on the rewards, with parameters that are all distinct and all different from $1$,
there exist two constants
$\beta > 0$ and $\gamma \geq 0$ and an ordering of
the considered distributions such that
\[
\E r_n \geq \beta \, e^{- \gamma n }~.
\]
\end{corollary}

Theorem~\ref{th:MainDistDep} is proved below and Corollary~\ref{cor:LBunif}
follows from the fact that the cumulative regret is always bounded by $n$.
To get further the point of the theorem, one should keep in mind that the typical (distribution-dependent)
rate of growth of the cumulative regret of good algorithms, e.g., UCB1 \cite{ACF02}, is $\varepsilon(n) = \ln n$.
This, as asserted in \cite{LR85}, is the optimal rate.
Hence a recommendation strategy based on such allocation strategy is bound to suffer a simple regret that decreases
at best polynomially fast. We state this result for the slight modification UCB$(\a)$ of UCB1
stated in Figure~\ref{fig:alloc} and introduced in \cite{AMS09};
its proof relies on noting that it achieves a cumulative regret bounded by
a large enough distribution-dependent constant times $\epsilon(n) = \a \ln n$.

\begin{corollary}[Distribution-dependent lower bound for UCB$(\a)$]
\label{cor:LBUCBp}
The allocation strategy $(I_t)$ given by the forecaster UCB$(\a)$ of Figure~\ref{fig:alloc} ensures that
for any recommendation strategy $(J_t)$ and
all sets of $K \geq 3$ Bernoulli distributions on the rewards, with parameters that are all distinct and all different from $1$,
there exist two constants
$\beta > 0$ and $\gamma \geq 0$ (independent of $\a$) and an ordering of
the considered distributions such that
\[
\E r_n \geq \beta \, n^{- \gamma \a }~.
\]
\end{corollary}

\begin{proof}
The intuitive version of the proof of Theorem~\ref{th:MainDistDep}
is as follows. The basic idea is to consider a tie case when the best and worst arms have zero empirical means; it happens often enough (with a probability at least exponential in the number of times we pulled these arms)
and results in the forecaster basically having to pick another arm and suffering some regret.
Permutations are used to control the case of untypical or naive forecasters that would despite all pull an arm with zero empirical mean, since they force a situation when those forecasters choose the worst arm instead of the best one.

Formally, we fix the forecaster (a pair of allocation and recommendation strategies) and a corresponding function $\epsilon$
such that the assumption of the theorem is satisfied. We denote by $\bp_n = (p_{1,n},\hdots,p_{K,n})$
the probability distribution from which $J_n$ is drawn at random thanks to an auxiliary distribution.
Note that $\bp_{n}$ is a random vector which depends on $I_1,\hdots,I_n$ as well as on the obtained rewards $Y_1,\ldots,Y_n$.
We consider below a set of $K \geq 3$ distinct Bernoulli distributions, satisfying the conditions of the theorem; actually, we
only use below that their parameters are (up to a first ordering) such that $1 > \mu_1 > \mu_2 \geq \mu_3 \geq \ldots \geq \mu_K \geq 0$ and
$\mu_2 > \mu_K$ (thus, $\mu_2 > 0$).

\medskip
\textbf{Step 0} introduces another layer of notation. The latter depends on permutations $\sigma$
of $\{ 1,\,\ldots,\,K\}$. To have a gentle start, we first describe the notation when
the permutation is the identity, $\sigma = \id$. We denote by $\P$ and $\E$
the probability and expectation with respect to the original $K$-tuple
$\nu_1,\ldots,\nu_K$ of distributions over the arms. For $i = 1$ (respectively, $i = K$), we denote by
$\P_{i,\id}$ and $\E_{i,\id}$ the probability and expectation with
respect to the $K$-tuples formed by $\delta_0,\nu_2,\ldots,\nu_K$
(respectively, $\delta_0,\nu_2,\ldots,\nu_{K-1},\delta_0$), where
$\delta_0$ denotes the Dirac measure on $0$.

For a given permutation $\sigma$,
we consider a similar notation up to a reordering, as follows.
The symbols $\P_{\sigma}$ and $\E_{\sigma}$ refer to the probability and expectation with
respect to the $K$-tuple of distributions over the arms formed by the $\nu_{\sigma^{-1}(1)},\ldots,
\nu_{\sigma^{-1}(K)}$. Note in particular that the $i$--th best arm is located in the $\sigma(i)$--th position.
Now, we denote for $i = 1$ (respectively, $i = K$) by
$\P_{i,\sigma}$ and $\E_{i,\sigma}$ the probability and expectation with
respect to the $K$-tuple formed by the $\nu_{\sigma^{-1}(i)}$, except that we replaced the
best of them, located in the $\sigma(1)$--th position, by a Dirac measure on 0 (respectively,
the best and worst of them, located in the $\sigma(1)$--th and $\sigma(K)$--th positions,
by Dirac measures on 0). We provide now a proof in six steps.

\medskip
\textbf{Step 1} lower bounds the quantity of interest
by an average of the simple regrets obtained by reordering,
\[
\max_{\sigma} \,\, \E_\sigma r_n \geq
\frac{1}{K!} \sum_{\sigma} \,\, \E_\sigma r_n
\geq \frac{\mu_1 - \mu_2}{K!} \,
\sum_{\sigma} \E_{\sigma} \! \left[ 1 - p_{\sigma(1),n} \right]~,
\]
where we used that under $\P_\sigma$, the index of the best arm is $\sigma(1)$
and the minimal regret for playing any other arm is at least $\mu_1 - \mu_2$.

\medskip
\textbf{Step 2} rewrites
each term of the sum over $\sigma$ as the product of three simple terms. We use
first that $\P_{1,\sigma}$ is the same as $\P_\sigma$,
except that it ensures that arm $\sigma(1)$ has zero reward throughout.
Denoting by
\[
C_{i,n} = \sum_{t=1}^{T_i(n)} X_{i,t}
\]
the cumulative reward of the $i$--th arm till round $n$, one then gets
\begin{eqnarray*}
\E_{\sigma} \bigl[ 1 - p_{\sigma(1),n} \bigr] & \geq &
\E_{\sigma} \! \left[ \left( 1 - p_{\sigma(1),n} \right) \ds1_{ \{ C_{\sigma(1),n} = 0 \} } \right] \\
& = & \E_{\sigma} \Bigl[ 1 - p_{\sigma(1),n} \,\, \big| \,\ C_{\sigma(1),n} = 0 \Bigr]
\times \, \P_\sigma \left\{ C_{\sigma(1),n} = 0 \right\} \\
& = & \E_{1,\sigma} \bigl[ 1 - p_{\sigma(1),n} \bigl]
\,\, \P_\sigma \left\{ C_{\sigma(1),n} = 0 \right\}~.
\end{eqnarray*}
Second, repeating the argument from $\P_{1,\sigma}$ to $\P_{K,\sigma}$,
\begin{eqnarray*}
\E_{1,\sigma} \bigl[ 1 - p_{\sigma(1),n} \bigl]
& \geq & \E_{1,\sigma} \Bigl[ 1 - p_{\sigma(1),n} \,\, \big| \,\ C_{\sigma(K),n} = 0 \Bigr]
\,\, \P_{1,\sigma} \left\{ C_{\sigma(K),n} = 0 \right\} \\
& = & \E_{K,\sigma} \bigl[ 1 - p_{\sigma(1),n} \bigl]
\,\, \P_{1,\sigma} \left\{ C_{\sigma(K),n} = 0 \right\}
\end{eqnarray*}
and therefore,
\begin{equation}
\label{eq:steps}
\E_{\sigma} \! \left[ 1 - p_{\sigma(1),n} \right] \geq
\E_{K,\sigma} \bigl[ 1 - p_{\sigma(1),n} \bigl]
\,\, \P_{1,\sigma} \! \left\{ C_{\sigma(K),n} = 0 \right\} \,\, \P_\sigma \! \left\{ C_{\sigma(1),n} = 0 \right\}~.
\end{equation}

\medskip
\textbf{Step 3} deals with the second term in the right-hand side of~(\ref{eq:steps}),
\[
\P_{1,\sigma} \left\{ C_{\sigma(K),n} = 0 \right\}
= \E_{1,\sigma} \! \left[ \left( 1 - \mu_K \right)^{T_{\sigma(K)}(n)} \right]
\geq \left( 1 - \mu_K \right)^{\E_{1,\sigma} T_{\sigma(K)}(n)}~,
\]
where the equality can be seen by conditioning on $I_1, \hdots, I_n$
and then taking the expectation, whereas the inequality is a consequence of Jensen's inequality.
Now, the expected number of times the suboptimal arm $\sigma(K)$ is pulled
under $\P_{1,\sigma}$ (for which $\sigma(2)$ is the optimal arm) is bounded by the regret,
by the very definition of the latter:
$( \mu_2 - \mu_K) \, \E_{1,\sigma} T_{\sigma(K)}(n) \leq \E_{1,\sigma} R_n$.
By hypothesis, there exists a constant $C$ such that for all $\sigma$,
$\,\, \E_{1,\sigma} R_n \leq C \, \epsilon(n)$;
the constant $C$ in the hypothesis of the theorem depends on the (order of the) distributions
but this can be circumvent by taking the maximum of $K!$ values to get the previous statement.
We finally get
\[
\P_{1,\sigma} \bigl\{  C_{\sigma(K),n} = 0 \bigr\}
\geq \left( 1 - \mu_K \right)^{C \epsilon(n) / (\mu_2 - \mu_K)}~.
\]

\medskip
\textbf{Step 4} lower bounds the third term in the right-hand side of~(\ref{eq:steps}) as
\[
\P_{\sigma} \bigl\{ C_{\sigma(1),n} = 0 \bigr\}
\geq \left( 1 - \mu_1 \right)^{C \epsilon(n) / \mu_2}~.
\]
We denote by $W_n = (I_1, Y_1, \ldots, I_n, Y_n)$ the history of pulled arms and
obtained payoffs up to time $n$.
What follows is reminiscent of the techniques used in~\cite{MT04}.
We are interested in certain realizations $w_n = (i_1, y_1, \ldots, i_n, y_n)$
of the history: we consider the subset $\mathcal{H}$ formed by the elements $w_n$
such that whenever $\sigma(1)$ was played, it got a null reward,
that is, such that $y_t = 0$ for all indexes $t$ with $i_t = \sigma(1)$.
For all arms $j$, we then denote by $t_{j}(w_n)$ the realization of $T_j(n)$ corresponding to $w_n$.
Since the likelihood of an element $w_n \in \mathcal{H}$ under $\P_\sigma$ is $(1-\mu_1)^{t_{\sigma(1)}(w_n)}$ times
the one under $\P_{1,\sigma}$, we get
\begin{multline}
\nonumber
\P_\sigma \bigr\{ C_{\sigma(1),n} = 0 \bigr\} = \sum_{w_n \in \mathcal{H}} \P_\sigma \left\{ W_n = w_n \right\}  \\
= \sum_{w_n \in \mathcal{H}} \left( 1 - \mu_1 \right)^{t_{\sigma(1)}(w_n)} \, \P_{1,\sigma} \left\{ W_n = w_n \right\}
= \E_{1,\sigma} \left[ \left( 1 - \mu_1 \right)^{T_{\sigma(1)}(n)} \right]~.
\end{multline}

The argument is concluded as before, first by Jensen's inequality and then,
by using that
$\mu_2 \, \E_{1,\sigma} T_{\sigma(1)}(n) \leq \E_{1,\sigma} R_n
\leq C \, \epsilon(n)$
by definition of the regret and the hypothesis put on its control.

\medskip
\textbf{Step 5} resorts to a
symmetry argument to show that as far as the first term
of the right-hand side of~(\ref{eq:steps}) is concerned,
\[
\sum_{\sigma} \E_{K,\sigma} \Bigl[ 1 - p_{\sigma(1),n} \Bigl] \geq \frac{K!}{2} .
\]
Since $\P_{K,\sigma}$ only depends on $\sigma(2), \ldots, \sigma(K-1)$,
we denote by $\P^{\sigma(2), \ldots, \sigma(K-1)}$ the common value of
these probability distributions when $\sigma(1)$ and $\sigma(K)$ vary (and
a similar notation for the associated expectation).
We can thus group the permutations $\sigma$ two by two according to these $(K-2)$--tuples,
one of the two permutations being defined by $\sigma(1)$ equal to one of the two elements of $\{1,\ldots,K\}$
not present in the $(K-2)$--tuple,
and the other one being such that $\sigma(1)$ equals the other such element. Formally,
\begin{eqnarray*}
\sum_{\sigma} \E_{K,\sigma} p_{\sigma(1),n}
& = & \sum\limits_{j_2,\ldots,j_{K-1}} \E^{j_2,\ldots,j_{K-1}} \left[ \sum\limits_{j \in \{1,\ldots,K\}
\setminus \{j_2,\ldots,j_{K-1} \}}
p_{j,n} \right] \\
& \leq & \sum\limits_{j_2,\ldots,j_{K-1}} \E^{j_2,\ldots,j_{K-1}} \bigl[ 1 \bigr] = \frac{K!}{2}~,
\end{eqnarray*}
where the summations over $j_2,\ldots,j_{K-1}$ are over all possible
$(K-2)$--tuples of distinct elements in $\{1,\ldots,K\}$.

\medskip
\textbf{Step 6} simply puts all pieces together and lower bounds $\ \ \displaystyle{\max_{\sigma} \,\, \E_\sigma r_n} \ \ $ by
\begin{eqnarray*}
& & \frac{\mu_1 - \mu_2}{K!} \,\, \sum_{\sigma} \, \E_{K,\sigma} \bigl[ 1 - p_{\sigma(1),n} \bigl]
\,\, \P_\sigma \left\{ C_{\sigma(1),n} = 0 \right\} \,\, \P_{1,\sigma} \left\{ C_{\sigma(K),n} = 0 \right\} \\
& \geq &
\frac{\mu_1 - \mu_2}{2} \, \left(
\left( 1 - \mu_K \right)^{C / (\mu_2 - \mu_K)} \, \left( 1 - \mu_1 \right)^{C / \mu_2}
\right)^{\epsilon(n)}~.
\end{eqnarray*}
\end{proof}

\section{Upper bounds on the simple regret}
\label{sec:UBS}

In this section, we aim at qualifying the implications of Theorem~\ref{th:MainDistDep} by pointing out
that is should be interpreted as a result for large $n$ only. For moderate values of $n$, strategies
not pulling each arm a linear number of times in the exploration phase can have a smaller simple regret.
To do so, we consider only two natural and well-used allocation strategies since the
aim of this paper is mostly to study the links between the cumulative and simple regret and not really
to prove the best possible bounds on the simple regret.
More sophisticated allocation strategies were considered recently in~\cite{AuBuMu10} and they can be used
to improve on the upper bounds on the simple regret presented below.

The first allocation strategy is the uniform allocation, which we use as a simple benchmark; it pulls
each arm a linear number of times (see Figure~\ref{fig:alloc} for its formal description).
The second one is UCB$(\a)$ (a variant of UCB1 introduced in \cite{AMS09} using an exploration rate parameter $\a>1$
and described also in Figure~\ref{fig:alloc}). It is designed for the classical exploration--exploitation dilemma
(i.e., it minimizes the cumulative regret) and pulls suboptimal arms a logarithmic number of times only.
\begin{figure}[h]
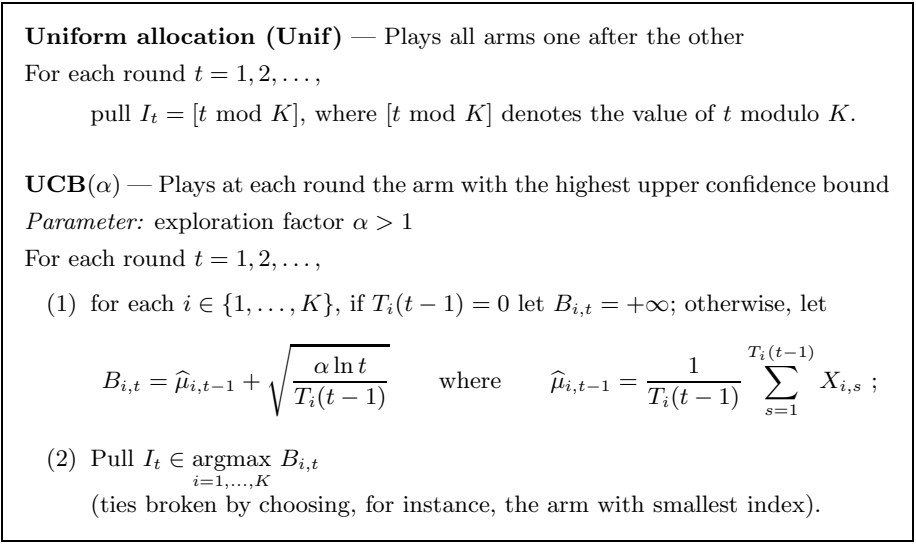

\bookbox{\small
\textbf{Uniform allocation (Unif)} --- Plays all arms one after the other

\smallskip \noindent
For each round $t=1,2,\ldots,$
\begin{itemize}
\item[] pull $I_t = \mod{t}{K}$, where $\mod{t}{K}$ denotes the value of $t$ modulo $K$.
\end{itemize}

\bigskip
\textbf{UCB$(\a)$} ---
Plays at each round the arm with the highest upper confidence bound

\smallskip
{\em Parameter:} exploration factor $\a>1$

\smallskip \noindent
For each round $t=1,2,\ldots,$

\begin{itemize}
\item[(1)]
for each $i \in \{1,\ldots,K\}$, if $T_i(t-1)=0$ let $B_{i,t}=+\infty$; otherwise, let
\[
B_{i,t}=\wh{\mu}_{i,t-1} + \sqrt{\frac{\a \ln t}{T_{i}(t-1)}}
\qquad \mbox{where} \qquad \wh{\mu}_{i,t-1} = \frac{1}{T_{i}(t-1)} \sum_{s = 1}^{T_{i}(t-1)} X_{i,s}~;
\]
\item[(2)] Pull $\displaystyle{I_t \in \argmax_{i = 1,\ldots,K} \, B_{i,t}}$ \\
(ties broken by choosing, for instance, the arm with smallest index).
\end{itemize}
}
\caption{\label{fig:alloc} Two allocation strategies.}
\end{figure}

In addition to these allocation strategies we consider three recommendation strategies,
the ones that recommend respectively the empirical distribution of plays, the empirical best arm, or the most played arm.
They are formally defined in Figure~\ref{fig:recomm}.
\begin{figure}[h]
\bookbox{\small
Parameters: the history $I_1,\ldots,I_n$ of played actions and of their associated rewards $Y_1,\ldots,Y_n$,
grouped according to the arms as $X_{i,1},\ldots,X_{i,T_{i}(n)}$, for $i = 1,\ldots,n$

\bigskip
\textbf{Empirical distribution of plays (EDP)} \\
Recommends arm $i$ with probability $T_i(n)/n$, that is, draws $J_n$ at random according to
\[
\bp_n = \left( \frac{T_1(n)}{n}, \, \ldots, \, \frac{T_K(n)}{n} \right)~.
\]

\bigskip
\textbf{Empirical best arm (EBA)} \\
Only considers arms $i$ with $T_{i}(n) \geq 1$, computes their associated empirical means
\[
\wh{\mu}_{i,n} = \frac{1}{T_{i}(n)} \sum_{s = 1}^{T_{i}(n)} X_{i,s}~,
\]
and forms the recommendation
\[
J_n \in \argmax_{i=1,\ldots,K} \, \wh{\mu}_{i,n}
\]
(ties broken in some way).

\bigskip
\textbf{Most played arm (MPA)} \\
Recommends the most played arm,
\[
J_n \in \argmax_{i=1,\ldots,K} \, T_{i}(n)
\]
(ties broken in some way).}
\caption{\label{fig:recomm} Three recommendation strategies.}
\end{figure}

\medskip
Table~\ref{tab:bounds} summarizes the distribution-dependent and distribution-free bounds we could prove
in this paper (the difference between the two families of bounds is whether the constants in the bounds
can depend or not on the unknown distributions $\nu_j$).
\begin{table}[h]
{\footnotesize
\begin{center}
\begin{tabular}{rcccc}
\toprule
 & & \multicolumn{3}{c}{Distribution-dependent}\\
 \cmidrule{3-5}
 & & EDP & EBA & MPA \\ \midrule
Uniform & & & $\bigcirc \, e^{-\bigcirc n}$ \, (Pr.\ref{prop:unif}) \, & \\
 UCB$(\a)$ & & $\bigcirc  (\a \ln n)/n$ \, (Rk.\ref{rk:origine2}) \, & $\bigcirc \, n^{- \bigcirc}$ \, (Rk.\ref{rk:KS}) \, &
 $\bigcirc \, n^{2(1-\a)}$ \, (Th.\ref{th:UCBp-MPA-1}) \, \\
\cmidrule{3-5}
Lower bound & & \multicolumn{3}{c}{$\bigcirc \, e^{-\bigcirc n}$ \, (Cor.\ref{cor:LBunif}) \,} \\
\midrule
 & & \multicolumn{3}{c}{Distribution-free} \\
 \cmidrule{3-5}
 & & EDP & EBA & MPA \\ \midrule
Uniform & & & $\displaystyle{\square \sqrt{\frac{K \ln K}{n}}}$ \, (Cor.\ref{cor:unif2}) \,& \\
 UCB$(\a)$ & & $\displaystyle{\square \sqrt{\frac{\a K \ln n}{n}}}$ \, (Rk.\ref{rk:origine2}) \,
 & $\displaystyle{\frac{\square}{\sqrt{\ln n}}}$ \, (Rk.\ref{rk:KS}) \, & $\displaystyle{\square \sqrt{\frac{\a K \ln n}{n}}}$ \,
 (Th.\ref{th:UCBp-MPA-2}) \, \\
\cmidrule{3-5}
Lower bound & & \multicolumn{3}{c}{$\displaystyle{\square \sqrt{\frac{K}{n}}}$ \, (Rk.\ref{rk:origine})} \, \\
\bottomrule
\end{tabular}
\end{center}
}
\caption{\label{tab:bounds}
Distribution-dependent (top) and distribution-free (bottom) upper bounds on the expected simple
regret of the considered pairs of allocation (rows) and recommendation (columns) strategies.
Lower bounds are also indicated. The $\square$ symbols denote the universal constants,
whereas the $\bigcirc$ are distribution-dependent constants. In parentheses, we provide
the reference within this paper (index of the proposition, theorem, remark, corollary)
where the stated bound is proved.}
\end{table}
It shows that two interesting couples of strategies are, on the one hand,
the uniform allocation together with the choice of the empirical best arm, and on the
other hand, UCB$(\a)$ together with the choice of the most played arm.
The first pair was perhaps expected, the second one might be considered more surprising.

Table~\ref{tab:bounds} also indicates that while for distribution-dependent bounds, the asymptotic optimal rate of decrease for the simple regret
in the number $n$ of rounds is exponential, for distribution-free bounds, this rate worsens to $1/\sqrt{n}$.
A similar situation arises for the cumulative regret, see
\cite{LR85} (optimal $\ln n$ rate for distribution-dependent bounds)
versus \cite{ACFS02} (optimal $\sqrt{n}$ rate for distribution-free bounds).

\begin{remark}
\label{rk:origine}
The distribution-free lower bound in Table~\ref{tab:bounds}
follows from a straightforward adaptation of the proof of the lower bound
on the cumulative regret in \cite{ACFS02}; one can prove that, for $n \geq K \geq 2$,
\[
\inf \sup \E r_n \geq \frac{1}{20} \sqrt{\frac{K}{n}}~,
\]
where the infimum is taken over all forecasters while the supremum
considers all sets of $K$ distributions over $[0,1]$. (The proof uses exactly
the same reduction to a stochastic setting as in \cite{ACFS02}. It is even
simpler than in the indicated reference since here, only what happens at round $n$
based on the information provided by previous rounds is to be considered; in the cumulative
case considered in \cite{ACFS02}, such an analysis had to be made at each round $t \leq n$.)
\end{remark}

\subsection{A simple benchmark: the uniform allocation strategy}
\label{sec:unif}

As explained above, the combination of the uniform allocation with the recommendation indicating the
empirical best arm, forms an important theoretical benchmark. This section
studies briefly its theoretical properties:
the rate of decrease of its simple regret is exponential in a distribution-dependent sense
and equals the optimal (up to a logarithmic term) $1/\sqrt{n}$ rate in the distribution-free case.

Below, we mean by the recommendation given by the empirical best arm at round $K\lfloor n/K \rfloor$
the recommendation $J_{K\lfloor n/K \rfloor}$ of EBA (see Figure~\ref{fig:recomm}),
where $\lfloor x \rfloor$ denotes the lower integer part of a real number $x$. The reason why at round $n$ we prefer
$J_{K\lfloor n/K \rfloor}$ to $J_n$ is only technical. The analysis is indeed simpler
when all averages over the rewards obtained by each arm are over the same number of terms.
This happens at rounds $n$ multiple of $K$ and this is why we prefer taking the recommendation of round
$K\lfloor n/K \rfloor$ instead of the one of round~$n$.

We propose first two distribution-dependent bounds, the first one is sharper
in the case when there are few arms, while the second one is suited for large $K$.

\begin{proposition}[Distribution-dependent; Unif and EBA]
\label{prop:unif}
The uniform allocation strategy associated with the recommendation given by the empirical best arm
(at round $K\lfloor n/K \rfloor$) ensures that
\[
\E r_n \leq \sum_{i : \Delta_i > 0} \Delta_i \, e^{ - \Delta_i^2 \lfloor n/K \rfloor}
\qquad \mbox{
for all} \ n \geq K~;
\]
and also, for all $\eta \in (0,1)$ and all
$\displaystyle{n \geq \max \left\{ K, \,\, \frac{K \ln K}{\eta^2 \Delta^2} \right\} }$,
\[
\E r_n \leq \left( \max_{i = 1,\ldots,K} \Delta_i \right)
\exp \left( - \frac{(1-\eta)^2}{2} \left\lfloor \frac{n}{K} \right\rfloor
\, \Delta^2 \right)~.
\]
\end{proposition}

\begin{proof}
To prove the first inequality, we relate the simple regret to the probability
of choosing a non-optimal arm,
\[
\E r_n = \E \Delta_{J_n} = \sum_{i : \Delta_i > 0} \Delta_i \,  \P\{J_n = i\}
\leq \sum_{i : \Delta_i > 0} \Delta_i \, \P \bigl\{ \wh{\mu}_{i,n} \geq \wh{\mu}_{i^*,n} \bigr\}
\]
where the upper bound follows from the fact that to be the empirical best arm,
an arm $i$ must have performed, in particular, better than a best arm $i^*$.
We now apply Hoeffding's inequality for independent bounded random variables, 
see~\cite{Hoe63}. The quantities $\wh{\mu}_{i,n} - \wh{\mu}_{i^*,n}$ are given by 
a (normalized) sum of $2 \lfloor n/K \rfloor$ random variables taking values in $[0,1]$ or in $[-1,0]$
and have expectation $- \Delta_i$.
Thus, the probability of interest is bounded by
\begin{multline}
\nonumber
\P \bigl\{ \wh{\mu}_{i,n} - \wh{\mu}_{i^*,n} \geq 0 \bigr\}
= \P \Bigl\{ \bigl( \wh{\mu}_{i,n} - \wh{\mu}_{i^*,n} \bigr) - \bigl( - \Delta_i \bigr)
\geq \Delta_i \Bigr\} \\
\leq \exp \left( - \frac{ 2 \, \Bigl( \left\lfloor n/K \right\rfloor \Delta_i \Bigr)^2}{2 \left\lfloor n/K \right\rfloor} \right)
= \exp \left( - \left\lfloor \frac{n}{K} \right\rfloor \Delta_i^2 \right)~,
\end{multline}
which yields the first result.

The second inequality is proved by resorting to a sharper concentration argument,
namely, the method of bounded differences, see \cite{McD89}, see also
\cite[Chapter~2]{DL01}. The complete proof can be found in Section~\ref{sec:app1}.
\end{proof}

The distribution-free bound of Corollary~\ref{cor:unif2}
is obtained not directly as a corollary of Proposition~\ref{prop:unif},
but as a consequence of its proof.
(It is not enough to optimize the bound of Proposition~\ref{prop:unif} over the $\Delta_i$,
for it would yield an additional multiplicative factor of $K$.)

\begin{corollary}[Distribution-free; Unif and EBA]
\label{cor:unif2}
The uniform allocation strategy associated with the recommendation given by the empirical best arm
(at round $K\lfloor n/K \rfloor$) ensures that
\[
\sup_{\nu_1,\ldots,\nu_K} \, \E r_n \leq 2 \,\sqrt{\frac{K \ln K}{n + K}}~,
\]
where the supremum is over all $K$--tuples $(\nu_1,\ldots,\nu_K)$ of distributions over $[0,1]$.
\end{corollary}

\begin{proof}
We extract from the proof of Proposition~\ref{prop:unif} that
\[
\P \{ J_n = i \} \leq \exp \left( - \left\lfloor \frac{n}{K} \right\rfloor \Delta_i^2 \right)~;
\]
we now distinguish whether a given $\Delta_i$ is more or less than a threshold
$\epsilon$, use that $\sum \P \{ J_n = i \} = 1$ and $\Delta_i \leq 1$ for all $i$, to write
\begin{eqnarray}
\E r_n = \sum_{i=1}^K \, \Delta_i \, \P\{J_n=i\}
\label{eq:infeps}
& \leq & \epsilon + \sum_{i : \Delta_i > \epsilon} \Delta_i \, \P\{J_n=i\} \\
\nonumber
& \leq & \epsilon + \sum_{i : \Delta_i > \epsilon} \Delta_i \, \exp \left( - \left\lfloor \frac{n}{K}
\right\rfloor \Delta_i^2 \right)~.
\end{eqnarray}
A simple study shows that
the function $x \in [0,1] \mapsto x \, \exp (- C x^2)$ is decreasing on $\bigl[ 1/\sqrt{2C},\,1 \bigr]$,
for any $C > 0$.
Therefore, taking $C = \lfloor n/K \rfloor$, we get that
whenever $\epsilon \geq 1 \big/ \sqrt{2 \lfloor n/K \rfloor}$,
\[
\E r_n \leq \epsilon + (K-1) \, \epsilon \, \exp \left( - \epsilon^2
\left\lfloor \frac{n}{K} \right\rfloor \right)~.
\]
Substituting $\epsilon = \sqrt{(\ln K)/\lfloor n/K \rfloor}$ concludes the proof.
\end{proof}

\subsection{Analysis of UCB$(\a)$ as an allocation strategy}

We start by studying the recommendation given by the most played arm.
A (distribution-dependent) bound is stated in Theorem~\ref{th:UCBp-MPA-1}; the bound does not involve any quantity depending on the
$\Delta_i$, but it only holds for rounds $n$ large enough, a statement that does involve the $\Delta_i$.
Its interest is first that it is simple to read, and second,
that the techniques used to prove it imply easily a second (distribution-free) bound, stated in Theorem~\ref{th:UCBp-MPA-2}
and which is comparable to Corollary~\ref{cor:unif2}.

\begin{theorem}[Distribution-dependent; UCB$(\a)$ and MPA]
\label{th:UCBp-MPA-1}
For $\a > 1$, the allocation strategy given by UCB$(\a)$ associated with the recommendation given by the most played arm
ensures that
\[
\E r_n \leq \frac{K}{\a-1} \, \left( \frac{n}{K} -1 \right)^{2(1-\a)}
\]
for all $n$ sufficiently large, {e.g.}, such that
$\displaystyle{n \geq K + \frac{4 K \a \ln n}{\Delta^2}}$ and $n \geq K(K+2)$.
\end{theorem}

The polynomial rate in the upper bound above is not a coincidence according to the lower bound exhibited in Corollary~\ref{cor:LBUCBp}.
Here, surprisingly enough, this polynomial rate of decrease is distribution-free
(but in compensation, the bound is only valid after a distribution-dependent time).
This rate illustrates Theorem~\ref{th:MainDistDep}: the larger $\a$, the larger
the (theoretical bound on the) cumulative regret of UCB$(\a)$
but the smaller the simple regret of UCB$(\a)$ associated with the recommendation given by the most played
arm.

\begin{theorem}[Distribution-free; UCB$(\a)$ and MPA]
\label{th:UCBp-MPA-2}
For $\a > 1$, the allocation strategy given by UCB$(\a)$ associated with the recommendation given by the most played arm
ensures that, for all $n \geq K(K+2)$,
\[
\sup_{\nu_1,\ldots,\nu_K} \,
\E r_n \leq \sqrt{\frac{4 K \a \ln n}{n-K}} + \frac{K}{\a-1} \, \left( \frac{n}{K} -1 \right)^{2(1-\a)}
= O \! \left(\sqrt{\frac{K \a \ln n}{n}} \right)~,
\]
where the supremum is over all $K$--tuples $(\nu_1,\ldots,\nu_K)$ of distributions over $[0,1]$.
\end{theorem}

\subsubsection{Proofs of Theorems~\ref{th:UCBp-MPA-1} and \ref{th:UCBp-MPA-2}}

We start by a technical lemma from which the two theorems will follow easily.

\begin{lemma}
\label{lm:UCBp-MPA}
Let $a_1,\ldots,a_K$ be real numbers such that $a_1 + \ldots + a_K = 1$ and $a_i \geq 0$ for all $i$, with
the additional property that for all suboptimal arms $i$ and all optimal arms $i^*$, one has $a_i \leq a_{i^*}$. Then for $\a > 1$, the allocation strategy given by UCB$(\a)$ associated with the recommendation given by the most played arm ensures that
\[
\E r_n \leq \frac{1}{\a-1} \sum_{i \ne i^*} (a_i n -1)^{2(1-\a)}
\]
for all $n$ sufficiently large, {e.g.}, such that, for all suboptimal arms $i$,
\[
a_i n \geq 1+\frac{4 \a \ln n}{\Delta_i^2} \quad \mbox{and} \quad
a_i n \geq K+2~.
\]
\end{lemma}

\begin{proof}
We first prove that whenever the most played arm $J_n$ is
different from an optimal arm $i^*$, then
at least one of the suboptimal arms $i$ is such that
$T_{i}(n) \geq a_i n$.
To do so, we use a contrapositive method and assume that
$T_{i}(n) < a_i n$ for all suboptimal arms. Then,
\[
\left( \sum_{i=1}^K a_i \right) n = n = \sum_{i = 1}^K T_i(n)
< \sum_{i^*} T_{i^*}(n) + \sum_{i} a_i n
\]
where, in the inequality, the first summation is over the optimal arms, the second one,
over the suboptimal ones.
Therefore, we get
\[
\sum_{i^*} a_{i^*} n < \sum_{i^*} T_{i^*}(n)
\]
and there exists at least one optimal arm $i^*$ such that
$T_{i^*}(n) > a_{i^*} n$. Since by definition of the vector $(a_1,\ldots,a_K)$,
one has $a_{i} \leq a_{i^*}$ for all suboptimal arms,
it comes that $T_i(n) < a_i n \leq a_{i^*} n < T_{i^*}(n)$ for all suboptimal arms,
and the most played arm $J_n$ is thus an optimal arm.

Thus, using that $\Delta_i \leq 1$ for all $i$,
\[
\E r_n = \E \Delta_{J_n} \leq
\sum_{i : \Delta_i > 0} \P \bigl\{ T_{i}(n) \geq a_i n \bigr\}~.
\]
A side-result extracted from \cite[proof of Theorem~7]{AMS09}, see also \cite[proof of Theorem~1]{ACF02},
states that for all suboptimal arms $i$ and all rounds $t \geq K+1$,
\begin{equation}
\label{eq:inegextr}
\P \Bigl\{ I_t = i \ \,\, \mbox{and} \,\, \ T_{i}(t-1) \geq \ell \Bigr\} \leq 2\,t^{1-2\a} \qquad
\mbox{whenever} \qquad \ell \geq \frac{4\a \ln n}{\Delta_i^2}~.
\end{equation}
We denote by $\lceil x \rceil$ the upper integer part of a real number $x$.
For a suboptimal arm $i$ and since by the assumptions on $n$ and the $a_i$, the choice
$\ell = \lceil a_i n \rceil - 1$ satisfies $\ell \geq K+1$ and $\ell \geq (4 \a \ln n)/\Delta_i^2$,
\begin{eqnarray}
\nonumber
\lefteqn{ \P \bigl\{ T_{i}(n) \geq a_i n \bigr\}
= \P \bigl\{ T_{i}(n) \geq \lceil a_i n \rceil \bigr\} } \\
\nonumber
& \leq & \sum_{t = \lceil a_i n \rceil}^{n} \P \Bigl\{ T_{i}(t-1) = \lceil a_i n \rceil - 1 \ \ \mbox{and} \ \ I_{t} = i \Bigr\} \\
\label{eq:controlMPA2}
& \leq & \sum_{t = \lceil a_i n \rceil}^{n} 2\,t^{1-2 \a} \leq 2 \int_{\lceil a_i n \rceil - 1}^\infty v^{1-2 \a}\,\mbox{d}v
\leq \frac{1}{\a-1} (a_i n - 1)^{2(1-\a)}~,
\end{eqnarray}
where we used a union bound for the second inequality and (\ref{eq:inegextr}) for the third inequality.
A summation over all suboptimal arms $i$ concludes the proof.
\end{proof}

\begin{proof}[of Theorem~\ref{th:UCBp-MPA-1}]
It consists in applying Lemma~\ref{lm:UCBp-MPA}
with the uniform choice $a_i = 1/K$
and recalling that $\Delta$ is the minimum of the $\Delta_i > 0$.
\end{proof}

\begin{proof}[of Theorem~\ref{th:UCBp-MPA-2}]
We start the proof by
using that $\sum \P \{ J_n=i \} = 1$ and $\Delta_i \leq 1$ for all $i$, and can thus write
\[
\E r_n = \E \Delta_{J_n} = \sum_{i=1}^K \, \Delta_i \, \P\{J_n=i\}
\leq \varepsilon + \sum_{i : \Delta_i > \varepsilon} \, \Delta_i \, \P\{J_n=i\}~.
\]
Since $J_n = i$ only if $T_i(n) \geq n/K$,
we get
\[
\E r_n \leq \varepsilon + \sum_{i : \Delta_i > \varepsilon}
\Delta_i \, \P \! \left\{ T_{i}(n) \geq \frac{n}{K} \right\}~.
\]
Applying (\ref{eq:controlMPA2}) with $a_i = 1/K$ leads to
$$\E r_n \leq \varepsilon + \sum_{i : \Delta_i > \varepsilon} \ \,\, \frac{\Delta_i}{\a-1} \,
\left( \frac{n}{K} - 1 \right)^{2(1-\a)}~,
$$
where $\varepsilon$ is chosen such that for all $\Delta_i > \varepsilon$, the condition
\[
\ell \geq n/K - 1 \geq (4\a \ln n)/\Delta_i^2
\]
is satisfied ($n/K - 1 \geq K+1$ being satisfied
by the assumption on $n$ and $K$). The conclusion thus follows
from taking, for instance,
\[
\varepsilon = \sqrt{(4 \a K \ln n)/(n-K)}
\]
and upper bounding all remaining $\Delta_i$ by 1.
\end{proof}

\subsubsection{Other recommendation strategies}

We discuss here the combination of UCB$(\a)$ with the two other recommendation strategies,
namely, the choice of the empirical best arm and the use of the empirical distribution of plays.

\begin{remark}[UCB$(\a)$ and EDP]
\label{rk:origine2}
We indicate in this remark from which results the corresponding bounds
of Table~\ref{tab:bounds} follow.
As noticed in the beginning of Section~\ref{sec:smallbad},
in the case of a recommendation formed by the empirical distribution of plays,
the simple regret is bounded in terms of the cumulative regret as
$\E r_n \leq \E R_n/n$. Now, the results in \cite{ACF02,AMS09} indicate that the cumulative regret
of UCB$(\a)$ is less than something of the form
\[
\bigcirc\, \a \ln n + \frac{3K}{2} + \frac{K}{2(\a-1)}~,
\]
where $\bigcirc$ denotes a constant dependent on $\nu_1,\ldots,\nu_K$.
The distribution-free bound on $\E R_n$ (and thus on $\E r_n$) follows
from the control, yielded by (\ref{eq:inegextr})
and a summation,
\[
\E T_i(n) \leq \frac{4 \a \ln n}{\Delta_i^2} + \frac{3}{2} + \frac{1}{2(\a-1)}~,
\]
together with the concavity argument
\begin{multline}
\nonumber
\E R_n = \sum_{i: \Delta_i > 0} \Delta_i \, \E T_i(n)
= \sum_{i: \Delta_i > 0} \left( \Delta_i \, \sqrt{\E T_i(n)} \right) \sqrt{\E T_i(n)} \\
\leq \sqrt{4 \a \ln n + \frac{3}{2} + \frac{1}{2(\a-1)}} \sum_{i: \Delta_i > 0} \sqrt{\E T_i(n)}
\leq \sqrt{\left(4 \a\ln n + \frac{3}{2} + \frac{1}{2(\a-1)} \right) Kn}~,
\end{multline}
where Jensen's inequality guaranteed that $\sum \sqrt{\E T_i(n)} \leq \sqrt{Kn}$.
\end{remark}

\begin{remark}[UCB$(\a)$ and EBA]
\label{rk:KS}
We can rephrase the results of \cite{KS06} as using UCB1 as an allocation strategy
and forming a recommendation according to the empirical best arm. In particular,
\cite[Theorem 5]{KS06} provides a distribution-dependent bound on the probability
of not picking the best arm with this procedure and can be used to derive the following bound on the simple
regret of UCB$(\a)$ combined with EBA: for all $n \geq 1$,
\[
\E r_n \leq \sum_{i : \Delta_i > 0} \frac{4}{\Delta_i} \, \left( \frac{1}{n} \right)^{\rho_{\a} \Delta_i^2/2}
\]
where $\rho_{\a}$ is a positive constant depending on $\a$ only. The leading constants $1/\Delta_i$
and the distribution-dependent exponent make it not as useful as the one
presented in Theorem~\ref{th:UCBp-MPA-1}.
The best distribution-free bound we could get from this bound was of the order of $1/\sqrt{\rho_{\a} \ln n}$,
to be compared to the asymptotic optimal $1/\sqrt{n}$ rate stated in Theorem~\ref{th:UCBp-MPA-2}.
\end{remark}

\section{Conclusions for the case of finitely many arms: Comparison of the bounds, simulation study}
\label{sec:comparison}

We first explain why, in some cases, the bound provided
by our theoretical analysis in Lemma~\ref{lm:UCBp-MPA} (for UCB$(\a)$ and MPA)
is better than the bound stated in Proposition~\ref{prop:unif} (for Unif and EBA).
The central point in the argument is that the bound of
Lemma~\ref{lm:UCBp-MPA} is of the form $\bigcirc \, n^{2(1-\a)}$,
for some distribution-dependent constant $\bigcirc$,
that is, it has a distribution-free convergence rate. In comparison,
the bound of Proposition~\ref{prop:unif} involves the gaps $\Delta_i$
in the rate of convergence.
Some care is needed in the comparison, since the bound for UCB$(\a)$ holds only
for $n$ large enough, but it is easy to find situations where for moderate values of $n$,
the bound exhibited for the sampling with UCB$(\a)$ is better than the one for the uniform
allocation. These situations typically involve a rather large number $K$ of arms;
in the latter case, the uniform allocation strategy only samples $\lfloor n/K\rfloor$ times each arm,
whereas the UCB strategy focuses rapidly its exploration on the best arms.
A general argument is proposed in Section~\ref{sec:app2} as well as a numerical example, showing that
for moderate values of $n$, the bounds associated with the sampling with UCB$(\a)$ are better than the ones
associated with the uniform sampling. This is further illustrated numerically, in the right part of
Figure~\ref{fig:1}).

\bigskip
To make short the longer story described in this paper, one can distinguish three regimes, according to the value of
the number of rounds $n$. The statements of these regimes (the ranges of their corresponding $n$)
involve distribution-dependent quantifications, to determine which $n$ are considered small, moderate, or large.
\begin{itemize}
\item For large values of $n$, uniform exploration is better
(as shown by a combination of the lower bound of Corollary~\ref{cor:LBUCBp} and of the upper bound of
Proposition~\ref{prop:unif}).
\item For moderate values of $n$, sampling with UCB$(\a)$ is preferable, as discussed just above (and in Section~\ref{sec:app2}).
\item For small values of $n$, little can be said and the best bounds to consider are perhaps the distribution-free bounds,
which are of the same order of magnitude for the two pairs of strategies.
\end{itemize}

\bigskip
We propose two simple experiments to illustrate our theoretical analysis;
each of them was run on $10^4$ instances of the problem and we plotted the average simple regret.
This is an instance of the Monte-Carlo method and provides accurate estimators of the expected
simple regret $\E r_n$.

The first experiment (upper plot of Figure~\ref{fig:1})
shows that for small values of $n$ (here, $n\leq 80$), the uniform allocation strategy can have an interesting behavior.
Of course the range of these ``small'' values of $n$ can be made arbitrarily large by decreasing the gap $\Delta$.
The second one (lower plot of Figure~\ref{fig:1}) corresponds to the numerical example to be described in Section~\ref{sec:app2}.
In both cases, the unclear picture for small values of $n$ become clearer for moderate values and shows an advantage
in favor of UCB--based allocation strategies. It also appears (here and in other non reported experiments)
that it is better in practice to use recommendations based on the empirical best arm rather than on the most played arm.
In particular, the theoretical upper bounds indicated in this paper for the combination of UCB as an allocation strategy and
the recommendation based on the empirical best arm (see Remark~\ref{rk:KS}) are probably to be improved.
\begin{figure}[p]
\begin{tabular}{c}
\includegraphics[scale = 0.8]{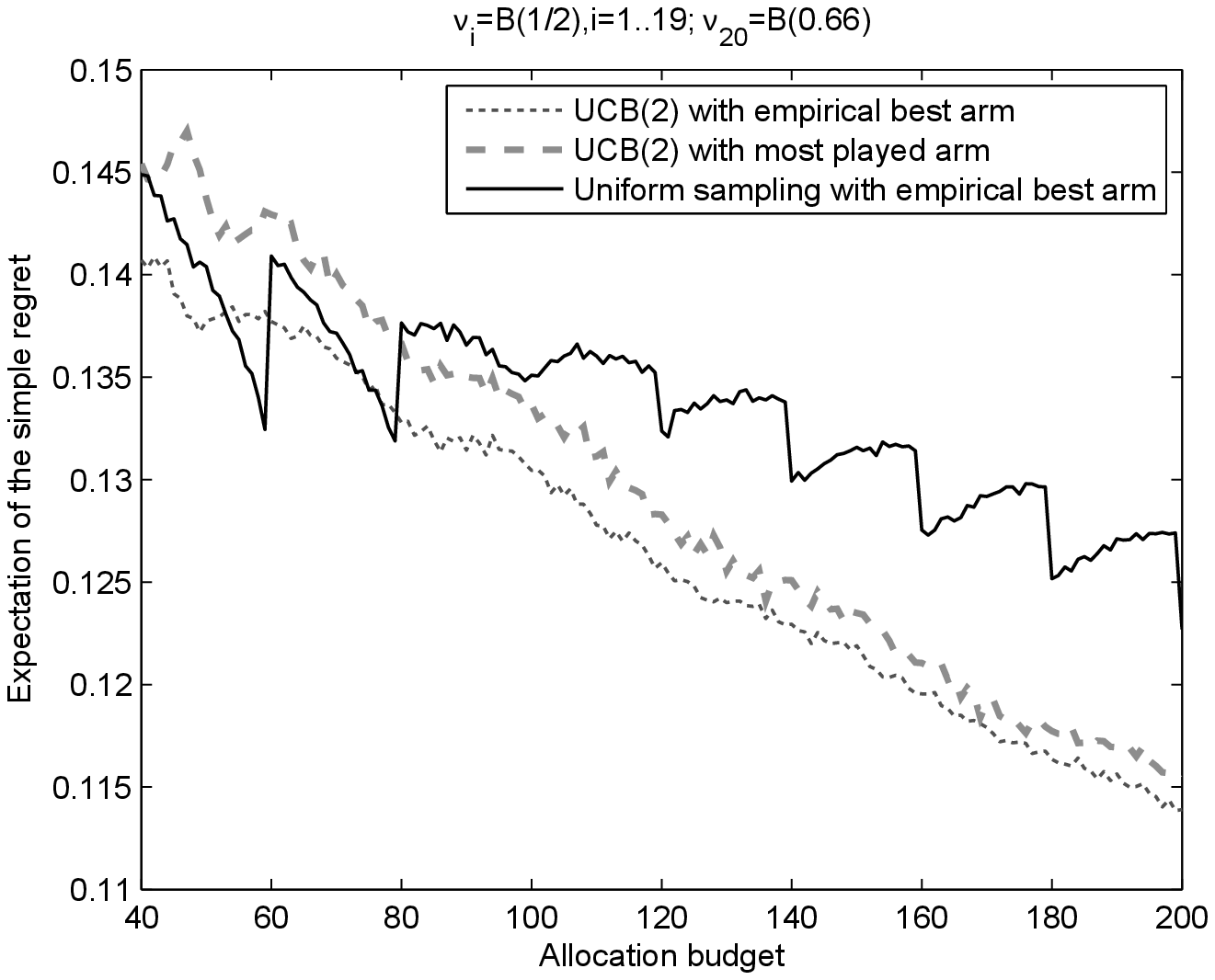} \\
\includegraphics[scale = 0.8]{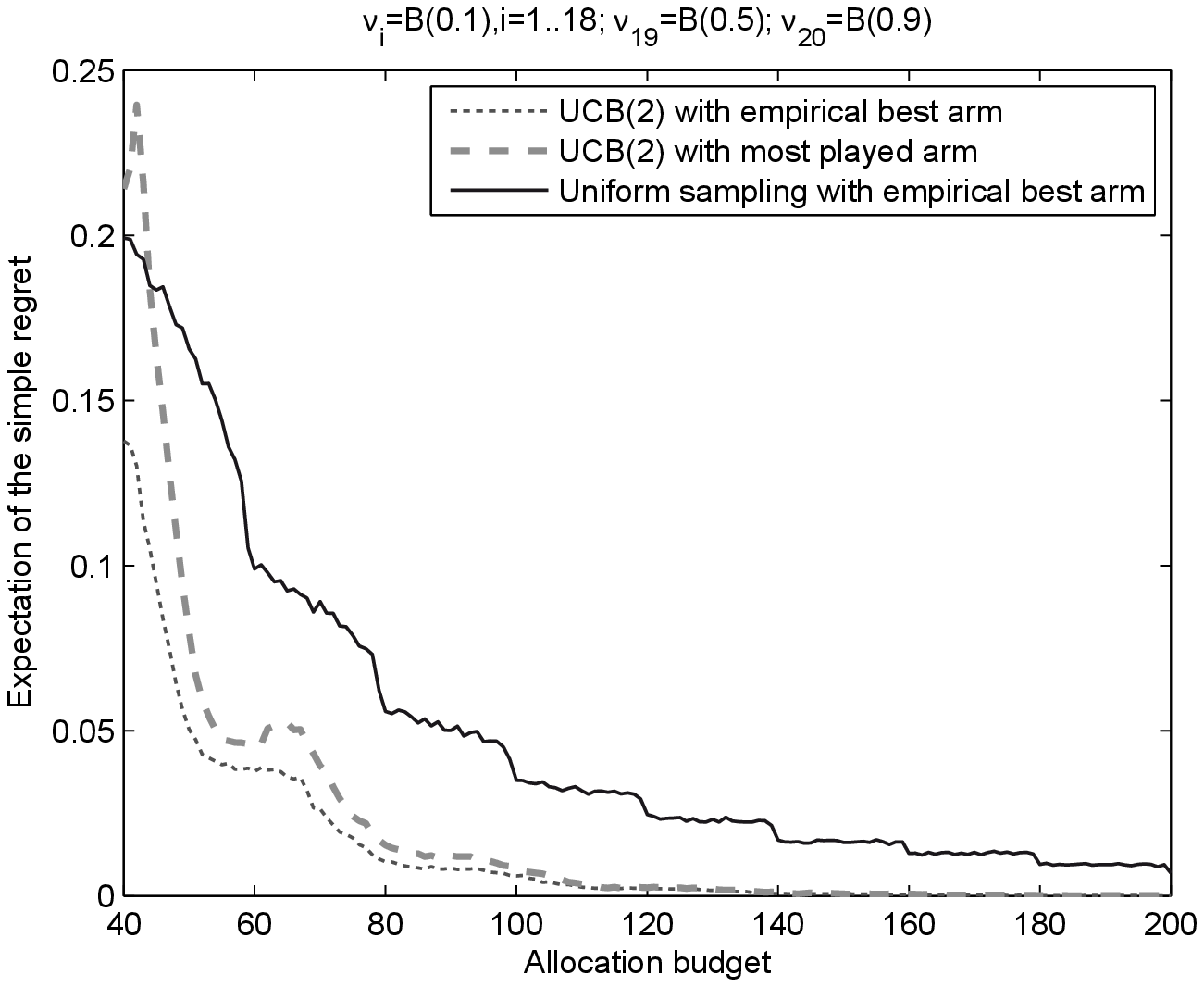}
\end{tabular}
\caption{\label{fig:1}
$K = 20$ arms with Bernoulli distributions of parameters indicated on top of each graph.
$x$-axis: number of rounds $n$; $y$-axis: simple regrets $\E r_n$ (estimated by a Monte-Carlo method).}
\end{figure}

\begin{remark}
We mostly illustrated here the small and moderate $n$ regimes.
This is because for large $n$, the simple regret is usually very small,
even below computer precision.
Therefore, because of the chosen ranges, we do not see yet the uniform allocation strategy getting better than UCB--based
strategies, a fact that is true however for large enough $n$.
This has an important impact on the interpretation of the lower bound of Theorem~\ref{th:MainDistDep}.
While its statement is in finite time, it should be interpreted as providing an asymptotic result only.
\end{remark}

\section{Pure exploration for continuous--armed bandits}
\label{sec:topo}

This section is of theoretical interest.
We consider the $\cX$--armed bandit problem already studied, e.g., in \cite{BMSS09,Kle04},
and (re)define the notions of cumulative and simple regret in this setting.
We show that the cumulative regret can be minimized if and only if the simple regret can be minimized,
and use this equivalence to characterize the metric spaces $\cX$ in which the cumulative
regret can be minimized: the separable ones. Here, in addition to its natural interpretation,
the simple regret thus appears as a tool for proving results on the cumulative regret.

\subsection{Description of the model of $\cX$--armed bandits}

We consider a bounded interval of $\R$, say $[0,1]$ again. We denote by $\cP([0,1])$
the set of probability distributions over $[0,1]$.
Similarly, given a topological space $\cX$,
we denote by $\cP(\cX)$ the set of probability distributions over $\cX$.
We then call environment on $\cX$
any mapping $E : \cX \to \cP([0,1])$.
We say that $E$ is continuous if the mapping that associates
to each $x \in \cX$ the expectation $\mu(x)$ of $E(x)$ is continuous;
we call the latter the mean-payoff function.

The $\cX$--armed bandit problem is described in Figures~\ref{fig:exp2cont}
and~\ref{fig:expcont}. There, an environment
$E$ on $\cX$ is fixed and we want various notions of regret to be small,
given this environment.
\begin{figure}[h]
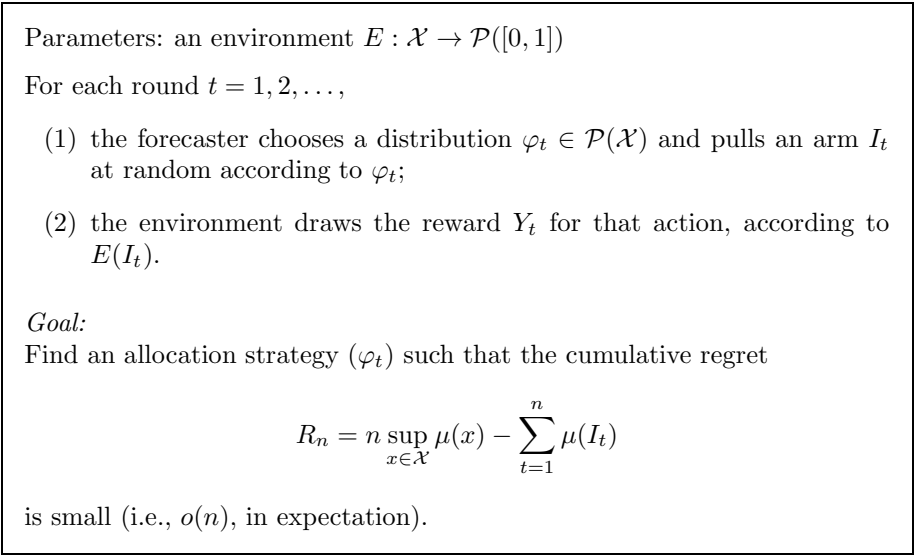

\bookbox{
Parameters: an environment $E : \cX \to \cP([0,1])$

\medskip\noindent
For each round $t=1,2,\ldots,$
\begin{itemize}
\item[(1)]
the forecaster chooses a distribution $\phi_t \in \cP(\cX)$
and pulls an arm $I_t$ at random according to $\phi_t$;
\item[(2)]
the environment draws the reward $Y_t$ for that action, according to $E(I_t)$.
\end{itemize}

\medskip
{\em Goal:} \\
Find an allocation strategy $(\phi_t)$ such that the cumulative regret
\[
R_n = n \sup_{x \in \cX} \mu(x) - \sum_{t=1}^n \mu(I_t)
\]
is small (i.e., $o(n)$, in expectation).
}
\caption{\label{fig:exp2cont} The classical
$\cX$--armed bandit problem.}
\end{figure}
\begin{figure}[h]
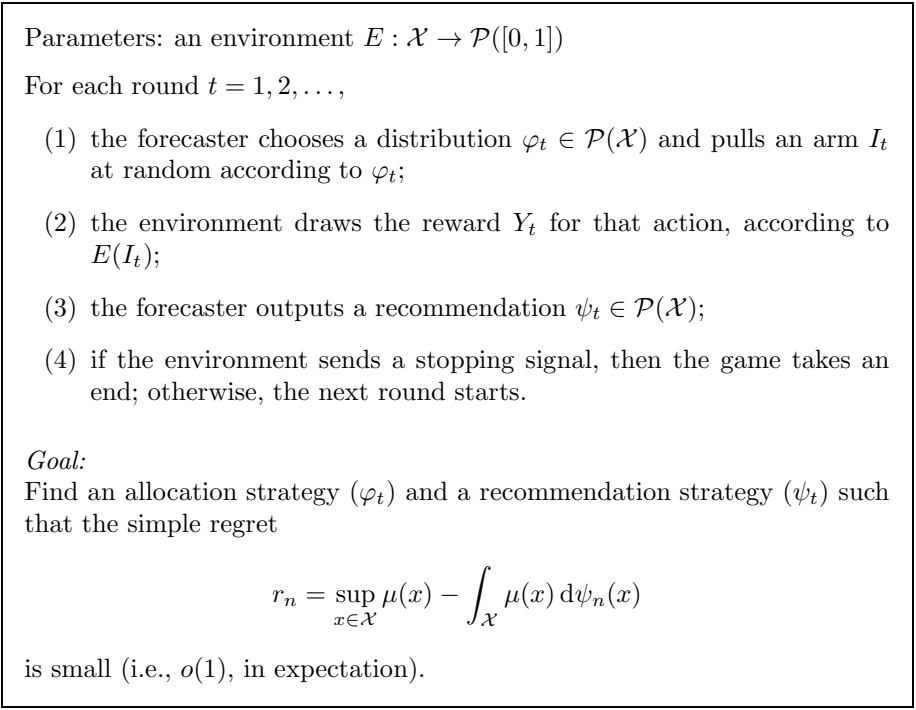

\bookbox{
Parameters: an environment $E : \cX \to \cP([0,1])$

\medskip\noindent
For each round $t=1,2,\ldots,$
\begin{itemize}
\item[(1)]
the forecaster chooses a distribution $\phi_t \in \cP(\cX)$
and pulls an arm $I_t$ at random according to $\phi_t$;
\item[(2)]
the environment draws the reward $Y_t$ for that action, according to $E(I_t)$;
\item[(3)]
the forecaster outputs a recommendation $\psi_t \in \cP(\cX)$;
\item[(4)]
if the environment sends a stopping signal, then the game takes an end;
otherwise, the next round starts.
\end{itemize}

\medskip
{\em Goal:} \\
Find an allocation strategy $(\phi_t)$ and a recommendation strategy $(\psi_t)$
such that the simple regret
\[
r_n = \sup_{x \in \cX} \mu(x) - \int_{\cX} \mu(x)\,\mbox{d}\psi_n(x)
\]
is small (i.e., $o(1)$, in expectation).
}
\caption{\label{fig:expcont}
The pure exploration problem for $\cX$--armed bandits.}
\end{figure}

We consider now families of environments and say that
a family $\cF$ of environments is explorable--exploitable (respectively, explorable) if there exists a forecaster such that
for any environment $E \in \cF$,
the expected cumulative regret $\E R_n$ (expectation taken with respect to $E$ and
all auxiliary randomizations) is $o(n)$ (respectively, $\E r_n = o(1)$).
Of course, explorability of $\cF$ is a milder
requirement than explorability--exploitability of $\cF$, as can be seen by
considering the recommendation given by the empirical distribution of plays of
Figure~\ref{fig:recomm} and applying the same argument as the one used at
the beginning of Section~\ref{sec:smallbad}.

In fact, it can be seen that the two notions are equivalent,
and this is why we will henceforth concentrate on explorability only,
for which characterizations as the ones of Theorem~\ref{th:CNSexplo}
are simpler to exhibit and prove.

\begin{lemma}
\label{lm:eqvexplo}
A family of environments $\cF$ is explorable if and only if it is explorable--exploitable.
\end{lemma}

The proof can be found in Section~\ref{sec:app3}. It relies essentially on designing a strategy suited for cumulative regret
from a strategy minimizing the simple regret; to do so, exploration and exploitation
occur at fixed rounds in two distinct phases and only the payoffs obtained during
exploration rounds are fed into the base allocation strategy.

\subsection{A positive result for metric spaces}

We denote by $\cP([0,1])^\cX$ the family of all possible environments $E$ on
$\cX$, and by $\cC \bigl( \cP([0,1])^\cX \bigr)$ the subset of
$\cP([0,1])^\cX$ formed by the continuous environments.

\begin{ex}
\label{ex2}
Previous sections were about the family $\cP([0,1])^\cX$
of all environments over $\cX = \{ 1,\ldots,K\}$
being explorable.
\end{ex}

The main result concerning $\cX$--armed bandit problems is
formed by the following equivalences in metric spaces. It generalizes
the result of Example~\ref{ex2}.

\begin{theorem}
\label{th:CNSexplo}
Let $\cX$ be a metric space. Then the family
$\cC \bigl( \cP([0,1])^\cX \bigr)$ is explorable if and only if
$\cX$ is separable.
\end{theorem}

\begin{corollary}
\label{cor:count}
Let $\cX$ be a set. The family
$\cP([0,1])^\cX$ is explorable if and only if $\cX$ is countable.
\end{corollary}

The proofs can be found in Section~\ref{sec:app4}.
Their main technical ingredient is that there exists a probability distribution
over a metric space $\cX$ giving a positive probability mass to all open sets if
and only if $\cX$ is separable. Then, whenever it exists, it allows some uniform
exploration.

\begin{remark}
We discuss here the links with results reported recently in~\cite{KlSl10}.
The latter restricts its attention to a setting where the space $\cX$
is a metric space (with metric denoted by $d$) and where the environments must
have mean-payoff functions that are $1$--Lipschitz with respect to $d$.
Its main concern is about the best achievable order of magnitude of the cumulative regret
with respect to $T$. In this respect, its main result is that a distribution-dependent
bound proportional to $\log(T)$ can be achieved if and only if the completion of $\cX$
is a compact metric space with countably many points. Otherwise, bounds on the regret
are proportional to at least $\sqrt{T}$. In fact, the links between our work and this article
are not in the statements of the results proved but rather in the techniques used in the proofs.
\end{remark}

\subsection{Proof of Lemma~\ref{lm:eqvexplo}}
\label{sec:app3}

\begin{proof}
In view of the comments before the statement of Lemma~\ref{lm:eqvexplo}, we need only to prove
that an explorable family $\cF$ is also explorable--exploitable. We consider a pair
of allocation $(\varphi_t)$ and recommendation $(\psi_t)$ strategies such that for all
environments $E \in \cF$, the simple regret satisfy $\E r_n = o(1)$, and provide a new
strategy $(\varphi'_t)$ such that its cumulative regret satisfies $\E R'_n = o(n)$ for all environments $E \in \cF$.

It is defined informally as follows. At round $t = 1$, it uses $\varphi'_1 = \varphi_1$ and gets a reward
$Y_1$. Based on this reward, the recommendation $\psi_1(Y_1)$ is formed and at round $t=2$, the new strategy
plays $\varphi'_2(Y_1) = \psi_1(Y_1)$. It gets a reward $Y_2$ but does not take it into account. It bases its choice
$\varphi'_3(Y_1,Y_2) = \varphi_2(Y_1)$ only on $Y_1$ and gets a reward $Y_3$. Based on $Y_1$ and $Y_3$,
the recommendation $\psi_2(Y_1,Y_3)$ is formed and played at rounds $t=4$ and $t=5$, i.e.,
\[
\varphi'_4(Y_1,Y_2,Y_3) = \varphi'_5(Y_1,Y_2,Y_3,Y_4) = \psi_2(Y_1,Y_3)~.
\]
And so on: the sequence of distributions chosen by the new strategy is formed using the applications
\begin{eqnarray*}
& \varphi_1, & \psi_1, \\
& \varphi_2, & \psi_2, \, \psi_2, \\
& \varphi_3, & \psi_3, \, \psi_3, \, \psi_3, \\
& \varphi_4, & \psi_4, \, \psi_4, \, \psi_4, \, \psi_4, \\
& \varphi_5, & \psi_5, \, \psi_5, \, \psi_5, \, \psi_5, \, \psi_5, \\
& \ldots &
\end{eqnarray*}

Formally, we consider regimes indexed by integers $t\geq 1$ and of length $1+t$. The $t$--th
regime starts at round
\[
1 + \sum_{s=1}^{t-1} (1+s) = t + \frac{t(t-1)}{2} = \frac{t(t+1)}{2}~.
\]
During this regime, the following distributions are used,
\begin{numcases}{\varphi'_{t(t+1)/2 + k} =}
\nonumber
\varphi_t \Bigl( \bigl( Y_{s(s+1)/2} \bigr)_{s = 1,\ldots,t-1} \Bigr) & if $k=0$; \\
\nonumber
\psi_t \Bigl( \bigl( Y_{s(s+1)/2} \bigr)_{s = 1,\ldots,t-1} \Bigr) & if $1 \leq k \leq t$.
\end{numcases}
Note that we only keep track of the payoffs obtained when $k=0$ in a regime.

The regret $R'_n$ at round $n$ of this strategy is as follows. We decompose $n$ in a unique manner
as
\begin{equation}
\label{eq:decompn}
n = \frac{t(n) \bigl( t(n) + 1 \bigr)}{2} + k(n) \qquad
\mbox{where} \qquad k(n) \in \bigl\{ 0,\ldots,t(n) \bigr\}~.
\end{equation}
Then (using also the tower rule),
\[
\E R'_n \leq t(n) + \Bigl( \E r_1 + 2 \, \E r_2 + \ldots + \bigl( t(n)-1 \bigr) \, \E r_{t(n)-1} + k(n) \, \E r_{t(n)} \Bigr)
\]
where the first term comes from the time rounds when the new strategy used the base allocation strategy to explore
and where the other terms come from the ones when it exploited.
This inequality can be rewritten as
\[
\frac{\E R'_n}{n} \leq \frac{t(n)}{n} + \frac{k(n) \, \E r_{t(n)} + \sum_{s=1}^{t(n)-1} s \, \E r_s}{n}~,
\]
which shows that $\E R'_n = o(n)$ whenever $\E r_s = o(1)$ as $s \to \infty$, since the first term in the right-hand side is of the order of $1/\sqrt{n}$ and the second one is a Cesaro average. This concludes that the exhibited strategy has a small cumulative
regret for all environments of the family, which is thus explorable--exploitable.
\end{proof}

\subsection{Proof of Theorem~\ref{th:CNSexplo} and its corollary}
\label{sec:app4}

The key ingredient is the following characterization of separability
(which relies on an application of Zorn's lemma); see, e.g., \cite[Appendix I, page 216]{Bil68}.

\begin{lemma}
\label{lm:Bil}
A metric space $\cX$, with distance denoted by $d$,
is separable if and only if
it contains no uncountable subset $A$
such that
\[
\rho = \inf \bigl\{ d(x,y): x,y \in A \bigr\} > 0~.
\]
\end{lemma}

Separability can then be characterized
in terms of the existence of a probability distribution with full support.
Though it seems natural, we did not see any reference to it in the literature and this is why we
state it.
(In the proof of Theorem~\ref{th:CNSexplo}, we will only use the straightforward direct part
of the characterization.)

\begin{lemma}
\label{lm:Gspace}
Let $\cX$ be a metric space.
There exists a probability distribution $\lambda$ on $\cX$ with
$\lambda(V) > 0$ for all open sets $V$
if and only if
$\cX$ is separable.
\end{lemma}

\begin{proof}
We prove the converse implication first.
If $\cX$ is separable, we denote by $x_1,x_2,\ldots$ a dense sequence. If it is finite with length $N$, we let
\[
\lambda = \frac{1}{N} \sum_{i=1}^N \delta_{x_i}
\]
and otherwise,
\[
\lambda = \sum_{i \geq 1} \frac{1}{2^i} \delta_{x_i}~.
\]
The result follows, since each open set $V$ contains at least some $x_i$.

For the direct implication, we use Lemma~\ref{lm:Bil} (and its notations).
If $\cX$ is not separable, then it
contains uncountably many disjoint open balls, formed by the $B(a,\rho/2)$, for
$a \in A$. If there existed a probability distribution $\lambda$
with full support on $\cX$, it would in particular give a positive probability
to all these balls; but this is impossible, since there are uncountably many of them.
\end{proof}

\subsubsection{Separability of $\cX$ implies explorability of the family $\cC \bigl( \cP([0,1])^\cX \bigr)$}

The proof of the converse part of the characterization provided
by Theorem~\ref{th:CNSexplo} relies on a somewhat uniform exploration that hits
each open set of $\cX$ after a random waiting time with distribution
depending on the probability of the open set.

\begin{proof}
Since $\cX$ is separable, there exists a probability distribution $\lambda$ on $\cX$
with $\lambda(V) > 0$ for all open sets $V$, as asserted by Lemma~\ref{lm:Gspace}.

The proposed strategy is then constructed in a way similar to the one exhibited in Section~\ref{sec:app2}, in the
sense that we also consider successives regimes, where the $t$--th of them has also length $1+t$.
They use the following allocations,
\begin{numcases}{\varphi_{t(t+1)/2 + k} =}
\nonumber
\lambda & if $k=0$; \\
\nonumber
\delta_{I_{k(k+1)/2}} & if $1 \leq k \leq t$.
\end{numcases}
Put in words, at the beginning of each regime, a new point $I_{t(t+1)/2}$ is drawn at random in $\cX$ according to $\lambda$,
and then, all previously drawn points $I_{s(s+1)/2}$, for $1 \leq s \leq t-1$, and the
new point $I_{t(t+1)/2}$ are pulled again, one after the other.

The recommendations $\psi_n$ are deterministic and put all probability mass on the best empirical arm
among the first played $g(n)$ arms (where the function $g$ will be determined by the analysis).
Formally, for all $x \in \cX$ such that
\[
T_n(x) = \sum_{t = 1}^n \mathbb{I}_{ \{ I_t = x \} } \geq 1~,
\]
one defines
\[
\wh{\mu}_n(x) = \frac{1}{T_n(x)} \sum_{t = 1}^n Y_t \, \mathbb{I}_{ \{ I_t = x \} }~.
\]
Then,
\[
\psi_n = \delta_{X_n^*} \quad \mbox{where} \quad X_n^* \in \argmax_{1 \leq s \leq g(n)} \,
\wh{\mu}_n \bigl( I_{s(s+1)/2} \bigr)
\]
(ties broken in some way, as usual;
and $g(n)$ to be chosen small enough so that all considered arms
have been played at least once). Note that exploration and exploitation appear in two distinct phases, as
was the case already, for instance, in Section~\ref{sec:unif}.

We now denote
\[
\mu^* = \sup_{x \in \cX} \mu(x) \quad \mbox{and} \quad
\mu^*_{g(n)} = \max_{1 \leq s \leq g(n)} \mu \bigl( I_{s(s+1)/2} \bigr)~;
\]
the simple regret can then be decomposed as
\[
\E r_n = \mu^* - \E \Bigl[ \mu \bigl( X_n^* \bigr) \Bigr] =
\left( \mu^* - \E \Bigl[ \mu^*_{g(n)} \Bigr] \right) +
\left( \E \Bigl[ \mu^*_{g(n)} \Bigr] - \E \Bigl[ \mu \bigl( X_n^* \bigr) \Bigr] \right)\!~,
\]
where the first difference can be thought of as an approximation error, and the second
one, as resulting from an estimation error. We now show that both differences
vanish in the limit.

We first deal with the approximation error. We fix $\varepsilon > 0$.
Since the mean-payoff function $\mu$ is continuous on $\cX$, there exists an open set
$V$ such that
\[
\forall x \in V, \qquad \mu^* - \mu(x) \leq \epsilon~.
\]
It follows that
\begin{multline}
\nonumber
\P \Bigl\{ \mu^* - \mu^*_{g(n)} > \varepsilon \Bigr\}
\leq \P \Bigl\{ \forall \, s \in \bigl\{ 1,\ldots,g(n) \bigr\}, \ \ \,\, I_{s(s+1)/2} \not\in V \Bigr\} \\
\leq \bigl( 1 - \lambda(V) \bigr)^{g(n)} \longrightarrow 0
\end{multline}
provided that $g(n) \to \infty$ (a condition that will be satisfied, see below).
Since in addition, $\mu^*_{g(n)} \leq \mu^*$, we get
\[
\limsup \ \mu^* - \E \Bigl[ \mu^*_{g(n)} \Bigr] \leq \epsilon~.
\]

For the difference resulting from the estimation error, we denote
\[
I_n^* \in \argmax_{1 \leq s \leq g(n)} \,
\mu \bigl( I_{s(s+1)/2} \bigr)
\]
(ties broken in some way).
Fix an arbitrary $\epsilon > 0$. We note that if for all $1 \leq s \leq g(n)$,
\[
\Bigl| \wh{\mu}_n \bigl( I_{s(s+1)/2} \bigr) - {\mu} \bigl( I_{s(s+1)/2} \bigr) \Bigr| \leq \epsilon~,
\]
then (together with the definition of $X_n^*$)
\[
\mu \bigl( X_n^* \bigr) \geq \wh{\mu}_n \bigl( X_n^* \bigr) - \epsilon
\geq \wh{\mu}_n \bigl( I_n^* \bigr) - \epsilon \geq
{\mu} \bigl( I_n^* \bigr) - 2\epsilon~.
\]
Thus, we have proved the inequality
\begin{equation}
\label{eq:unionbd}
\E \Bigl[ \mu^*_{g(n)} \Bigr] - \E \Bigl[ \mu \bigl( X_n^* \bigr) \Bigr] \leq 2 \epsilon + \P \biggl\{ \exists\,s \leq g(n),
\Bigl| \wh{\mu}_n \bigl( I_{s(s+1)/2} \bigr) - {\mu} \bigl( I_{s(s+1)/2} \bigr) \Bigr| > \epsilon
\biggr\}~.
\end{equation}
We use a union bound and control each (conditional) probability
\begin{equation}
\label{eq:cdprob}
\P \biggl\{ \Bigl| \wh{\mu}_n \bigl( I_{s(s+1)/2} \bigr) - {\mu} \bigl( I_{s(s+1)/2} \bigr) \Bigr| > \epsilon
\ \ \bigg| \ \ \cA_n \biggr\}
\end{equation}
for $1 \leq s \leq g(n)$, where
$\cA_n$ is the $\sigma$--algebra generated by the randomly drawn points
$I_{k(k+1)/2}$, for those $k$ with $k(k+1)/2 \leq n$. Conditionally to them,
$\wh{\mu}_n \bigl( I_{s(s+1)/2} \bigr)$ is an average of a deterministic number
of summands, which only depends on $s$, and thus, classical concentration-of-the-measure
arguments can be used.
For instance, the quantities (\ref{eq:cdprob}) are bounded, via an application of
Hoeffding's inequality \cite{Hoe63}, by
\[
2 \exp \Bigl( - 2 \, T_n \bigl( I_{s(s+1)/2} \bigr) \, \epsilon^2 \Bigr)~.
\]
We lower bound $T_n \bigl( I_{s(s+1)/2} \bigr)$.
The point $I_{s(s+1)/2}$ was pulled twice in regime $s$, once in each regime $s+1,\ldots,t(n)-1$,
and maybe in $t(n)$, where $n$ is decomposed again as in (\ref{eq:decompn}). That is,
\[
T_n \bigl( I_{s(s+1)/2} \bigr) \geq t(n) - s + 1 \geq \sqrt{2n} - 1 - g(n)~,
\]
since we only consider $s \leq g(n)$ and since (\ref{eq:decompn}) implies that
\[
n \leq \frac{t(n) \, \bigl( t(n) + 3 \bigr)}{2} \leq \frac{\bigl( t(n) + 2 \bigr)^2}{2}~, \quad \mbox{that is,} \ \ t(n) \geq \sqrt{2n} - 2~.
\]
Substituting this in the Hoeffding's bound, integrating, and taking a union bound lead from
(\ref{eq:unionbd}) to
\[
\E \Bigl[ \mu^*_{g(n)} \Bigr] - \E \Bigl[ \mu \bigl( X_n^* \bigr) \Bigr]
\leq 2 \epsilon + 2 g(n) \, \exp \left( - 2 \, \bigl( \sqrt{2n} - 1 - g(n) \bigr) \, \epsilon^2 \right)~.
\]
Choosing for instance $g(n) = \sqrt{n}/2$ ensures that
\[
\limsup \ \E \Bigl[ \mu^*_{g(n)} \Bigr] - \E \Bigl[ \mu \bigl( X_n^* \bigr) \Bigr] \leq 2\epsilon~.
\]

Summing up the two superior limits, we finally get
\[
\limsup \ \E r_n \leq
\limsup \ \mu^* - \E \Bigl[ \mu^*_{g(n)} \Bigr]  +
\limsup \ \E \Bigl[ \mu^*_{g(n)} \Bigr] - \E \Bigl[ \mu \bigl( X_n^* \bigr) \Bigr]
\leq 3\epsilon~;
\]
since this is true for all arbitrary $\epsilon > 0$, the proof is concluded.
\end{proof}

\subsubsection{Explorability of the family $\cC \bigl( \cP([0,1])^\cX \bigr)$ implies separability of $\cX$}

We now prove the direct part of the characterization provided
by Theorem~\ref{th:CNSexplo}. It basically follows from the impossibility of a uniform exploration, as
asserted by Lemma~\ref{lm:Gspace}.

\begin{proof}
Let $\cX$ be a non-separable metric space with metric denoted by $d$.
Let $A$ be an arbitrary uncountable subset of $\cX$ and let $\rho > 0$ be defined as
in Lemma~\ref{lm:Bil}; in particular, the balls $B(a,\rho/2)$ are disjoint,
for $a \in A$.

We now consider the subset of $\cC \bigl( \cP([0,1])^\cX \bigr)$ formed by the environments $E_a$
defined as follows. They are indexed by $a \in A$ and their corresponding
mean-payoff functions are given by
\[
\mu_a : x \in \cX \longmapsto \left( 1 - \frac{d(x,a)}{\rho/2} \right)^+~.
\]
The associated environments $E_a$ are deterministic, in the sense that they are defined as $E_a(x) = \delta_{\mu_a(x)}$.
Note that each $\mu_a$ is continuous, that $\mu_a(x) > 0$ for all $x \in B(a,\rho/2)$ but
$\mu_a(x) = 0$ for all $x \in \cX \setminus B(a,\rho/2)$;
that the best arm under $E_a$ is $a$ and that its gets a reward equal to
$\mu_a^* = \mu_a(a) = 1$.

We fix a forecaster and
denote by $\E_a$ the expectation under environment $E_a$ with respect with the auxiliary randomizations used
by the forecaster. Since $\mu_a$ vanishes outside $\bigl( B(a,\rho/2) \bigr)$ and has a maximum equal to $1$,
$$\E_a r_n = 1 - \E_a \left[ \int_\cX \mu_a(x) \,\mbox{d}\psi_n(x) \right]
\geq 1 - \E_a \Bigl[ \psi_n \bigl( B(a,\rho/2) \bigr) \Bigr]~.$$
We now show the existence of a non-empty set $A'$ such that for all $a \in A'$ and $n \geq 1$,
\begin{equation}
\label{eq:Brhoa}
\E_a \Bigl[ \psi_n \bigl( B(a,\rho/2) \bigr) \Bigr] = 0~;
\end{equation}
this indicates that $\E_a r_n = 1$ for all $n \geq 1$ and $a \in A'$, thus
preventing in particular $\cC \bigl( \cP([0,1])^\cX \bigr)$ from being explorable
by the fixed forecaster.

\medskip
The set $A'$ is constructed by studying the behavior of the forecaster under the
environment $E_0$ yielding deterministic null rewards throughout the space, i.e.,
associated with the mean-payoff function $x \in \cX \mapsto \mu_0(x) = 0$.
In the first round, the forecaster chooses a deterministic distribution $\phi_1 = \phi_1^0$ over
$\cX$, picks $I_1$ at random according to $\phi^0_1$, gets a deterministic payoff
$Y_1 = 0$, and finally recommends $\psi_1^0(I_1) = \psi_1(I_1,Y_1)$
(which depends on $I_1$ only, since the obtained payoffs are all null in a deterministic way).
In the second round, it chooses an allocation $\psi_2^0(I_1)$ (that depends only on $I_1$, for the same reasons as
before), picks $I_2$ at random according to $\psi^0_2(I_1)$, gets a null reward, and recommends $\psi^0_2(I_1,I_2)$;
and so on.

We denote by $\mathbb{A}$ the probability distribution giving the auxiliary randomizations used to draw the
$I_t$ at random, and for all integers $t$ and all measurable applications
\[
\nu : (x_1,\ldots,x_t) \in \cX^t \longmapsto \nu(x_1,\ldots,x_t) \in \cP(\cX)
\]
we introduce the distributions $\mathbb{A} \cdot \nu \in \cP(\cX)$ defined as the following
mixture of distributions.
For all measurable sets $V \subseteq \cX$,
\[
\mathbb{A} \cdot \nu (V) = \E_\mathbb{A} \left[ \int_\cX \mathbb{I}_V \,\mbox{d}\nu(I_1,\ldots,I_t) \right]~.
\]
A probability distribution can only put a positive mass on an at most countable number of disjoint sets.
Therefore, let $B_n$ and $C_n$ be defined as the at most countable sets of $a$ such that, respectively,
$\mathbb{A}\cdot\phi^0_n$ and $\mathbb{A}\cdot\psi^0_n$ give a positive probability mass to $B(a,\rho/2)$.
Then, let
\[
A' = A \setminus \left( \bigcup_{n \geq 1} B_n \, \cup \, \bigcup_{n \geq 1} C_n \right)
\]
be the uncountable, thus non empty, set of those elements of $A$ which are in no $B_n$ or $C_n$.

By construction, for all $a \in A'$, the forecaster
only gets null rewards; this is because $a$ is in no $B_n$ and therefore, with probability 1,
none of the $\varphi_n^0$ hits $B(a,\rho/2)$, which is exactly the set of those elements
of $\cX$ for which $\mu_a > 0$.
As a consequence, the forecaster behaves similarly under the environments $E_a$ and $E_0$, which means
that for all measurable sets $V \subseteq \cX$ and all $n \geq 1$,
\[
\E_a \bigl[ \phi_n (V) \bigr] = \mathbb{A} \cdot \phi_n^0 (V)
\quad \mbox{and} \quad
\E_a \bigl[ \psi_n (V) \bigr] = \mathbb{A} \cdot \psi_n^0 (V)~.
\]
In particular, since $a$ is in no $C_n$,
it hits in no recommendation $\psi^0_n$
the ball $B(a,\rho/2)$, which is exactly what remained to be proved, see~(\ref{eq:Brhoa}).
\end{proof}

\subsubsection{The countable case of Corollary~\ref{cor:count}}

We adopt an ``{\`a} la Bourbaki'' approach and derive this special case from the
general theory.

\begin{proof}
We endow $\cX$ with the discrete topology, i.e., choose the distance
\[
d(x,y) = \mathbb{I}_{ \{ x \ne y \} }~.
\]
Then, all applications defined on $\cX$ are continuous; in particular,
\[
\cC \bigl( \cP([0,1])^\cX \bigr) = \cP([0,1])^\cX~.
\]
In addition, $\cX$ is then separable if and only if it is countable.
The result thus follows immediately from Theorem~\ref{th:CNSexplo}.
\end{proof}

\subsection{An additional remark about uniform bounds}

In this paper, we mostly consider non-uniform bounds
(bounds that are individual as far as the environments are concerned).
As for uniform bounds, i.e., bounds on quantities of the form
\[
\sup_{E \in \cF} \E R_n \qquad \mbox{or} \qquad
\sup_{E \in \cF} \E r_n
\]
for some family $\cF$, two observations can be made.

First, it is easy to see that no sublinear uniform bound can be obtained for the
family of all continuous environments, as soon as there exists infinitely many disjoint open balls.

However one can exhibit such sublinear uniform bounds in some specific scenarios;
for instance, when $\cX$ is totally bounded and $\cF$ is formed by continuous functions
with a common bounded Lipschitz constant.

\section*{Acknowledgements}

The authors acknowledge support by
the French National Research Agency (ANR)
under grants 08-COSI-004 ``Exploration--exploitation for efficient resource allocation'' (EXPLO/RA)
and JCJC06-137444 ``From applications to theory in learning and adaptive statistics'' (ATLAS), as well as
by the PASCAL Network of Excellence under EC grant {no.} 506778.

An extended abstract of the present paper appeared in the \emph{Proceedings of the 20th International
Conference on Algorithmic Learning Theory} (ALT'09).

\appendix
\newpage

\section{Appendix}

\subsection{Proof of the second statement of Proposition~\ref{prop:unif}}
\label{sec:app1}

We use below the notations introduced in the proof of the first statement of
Proposition~\ref{prop:unif}.

\begin{proof}
Since some regret is suffered only
when an arm with suboptimal expectation has the best empirical performance,
$$\E r_n \leq \left( \max_{i = 1,\ldots,K} \Delta_i \right) \P \left\{ \max_{i : \Delta_i > 0}
\wh{\mu}_{i,n} \geq \wh{\mu}_{i^*,n} \right\}~.$$
Now, the quantity of interest can be rewritten as
\[
\left\lfloor \frac{n}{K} \right\rfloor \left( \max_{i : \Delta_i > 0} \wh{\mu}_{i,n} - \wh{\mu}_{i^*,n} \right)
= f \! \left( \vec{X}_1, \ldots, \vec{X}_{\left\lfloor n/K \right\rfloor} \right)
\]
for some function $f$, where for all $s = 1,\ldots,\lfloor n/K \rfloor$, we denote by
$\vec{X}_s$ the vector $(X_{1,s},\ldots,X_{K,s})$.
(The function $f$ is defined as a maximum of at most $K-1$ sums of differences.)
We apply the method of bounded differences, see \cite{McD89}, see also
\cite[Chapter~2]{DL01}. It is straightforward
that, since all random variables of interest take values in $[0,1]$,
the bounded differences condition is satisfied with ranges all equal to 2.
Therefore, the indicated concentration inequality states that
$$\P \left\{ \left( \max_{i : \Delta_i > 0}
\wh{\mu}_{i,n} - \wh{\mu}_{i^*,n} \right)
- \E \left[ \max_{i : \Delta_i > 0}
\wh{\mu}_{i,n} - \wh{\mu}_{i^*,n} \right] \geq \epsilon
\right\}
\leq \exp \left( - \frac{2 \left\lfloor n/K \right\rfloor \epsilon^2}{4} \right)$$
for all $\epsilon > 0$. We choose
$$\epsilon = - \E \left[ \max_{i : \Delta_i > 0} \wh{\mu}_{i,n} - \wh{\mu}_{i^*,n} \right]
\geq \min_{i : \Delta_i > 0} \Delta_i - \E \left[ \max_{i : \Delta_i > 0} \bigl\{
\wh{\mu}_{i,n} - \wh{\mu}_{i^*,n} + \Delta_i \bigr\} \right]$$
(where we used that the maximum of $K$ first quantities
plus the minimum of $K$ other quantities is less than the maximum
of the $K$ sums). We now argue that
\[
\E \left[ \max_{i : \Delta_i > 0} \bigl\{
\wh{\mu}_{i,n} - \wh{\mu}_{i^*,n} + \Delta_i \bigr\} \right] \leq
\sqrt{\frac{\ln K}{\lfloor n/K \rfloor}}~;
\]
this is done by a classical argument, using bounds on
the moment generating function of the random variables of interest.
Consider 
\[
Z_i = \lfloor n/K \rfloor \bigl( \wh{\mu}_{i,n} - \wh{\mu}_{i^*,n} + \Delta_i \bigr)
\]
for all $i = 1,\ldots,K$; they correspond to centered sums of $2 \lfloor n/K \rfloor$
independent random variables taking values in $[0,1]$ or $[-1,0]$.
Hoeffding's lemma (see, e.g.,
\cite[Chapter~2]{DL01}) thus imply that for all $\lambda > 0$,
\[
\E \left[ e^{\lambda Z_i} \right] \leq \exp \left( \frac{1}{8} \lambda^2 \,\, 2 \lfloor n/K \rfloor \right)
= \exp \left( \frac{1}{4} \lambda^2 \lfloor n/K \rfloor \right)~.
\]
A well-known inequality for maxima of subgaussian random variables (see~\cite[Chapter 2]{DL01}) then yields
\[
\E \left[ \max_{i=1,\ldots,K} Z_i \right] \leq \sqrt{\lfloor n/K \rfloor \ln K}~,
\]
which leads to the claimed upper bound.
Putting things together, we get that for the choice
$$\epsilon = - \E \left[ \max_{i : \Delta_i > 0} \wh{\mu}_{i,n} - \wh{\mu}_{i^*,n} \right]
\geq \min_{i : \Delta_i > 0} \Delta_i - \sqrt{\frac{\ln K}{\lfloor n/K \rfloor}} > 0$$
(for $n$ sufficiently large, a statement made precise below),
we have
\begin{eqnarray*}
\P \left\{ \max_{i : \Delta_i > 0} \wh{\mu}_{i,n} \geq \wh{\mu}_{i^*,n} \right\}
& \leq & \exp \left( - \frac{2 \left\lfloor n/K \right\rfloor \epsilon^2}{4} \right) \\
& \leq & \exp \left( - \frac{1}{2} \left\lfloor \frac{n}{K} \right\rfloor \left(
\min_{i : \Delta_i > 0} \Delta_i - \sqrt{\frac{\ln K}{\lfloor n/K \rfloor}} \, \right)^2 \right)~.
\end{eqnarray*}
The result follows for $n$ such that
\[
\min_{i : \Delta_i > 0} \Delta_i - \sqrt{\frac{\ln K}{\lfloor n/K \rfloor}}
\geq (1-\eta) \min_{i : \Delta_i > 0} \Delta_i~;
\]
the second part of the statement of Proposition~\ref{prop:unif} indeed only considers such $n$.
\end{proof}

\subsection{Detailed discussion of the heuristic arguments presented in Section~\ref{sec:comparison}}
\label{sec:app2}

We first state the following corollary to Lemma~\ref{lm:UCBp-MPA}.

\begin{theorem}
\label{th:UCBp-MPA-0}
The allocation strategy given by UCB$(\a)$ (where $\a > 1$) associated with the recommendation given by the most played arm
ensures that
\[
\E r_n \leq \frac{1}{\a-1} \sum_{i \ne i^*} \left( \frac{\beta n}{\Delta_i^2} - 1\right)^{2(1-\a)}
\]
for all $n$ sufficiently large, {e.g.}, such that
\[
\frac{n}{\ln n} \geq \frac{4\a+1}{\beta} \quad \mbox{and} \quad
n \geq \frac{K+2}{\beta} ( \Delta' )^2~,
\]
where
$\Delta' = \max_i \Delta_i$ and we denote by $K^*$ the number of optimal arms and
\[
\beta = \frac{1}{\displaystyle{\frac{K^*}{\Delta^2} + \sum_{i \ne i^*} \frac{1}{\Delta_i^2}}}~.
\]
\end{theorem}

\begin{proof}
We apply Lemma~\ref{lm:UCBp-MPA}
with the choice $a_i = \beta/\Delta_i^2$
for all suboptimal arms $i$ and $a_{i^*} = \beta/\Delta^2$
for all optimal arms $i^*$, where $\beta$ denotes the
normalization constant.
\end{proof}

For illustration, consider the case when there is one optimal arm, one $\Delta$--suboptimal arm and $K-2$ arms that
are $2\Delta$--suboptimal. Then
\[
\frac{1}{\beta} = \frac{2}{\Delta^2}+\frac{K-2}{(2\Delta)^2}=\frac{6+K}{4\Delta^2}~,
\]
and the previous bound of Theorem~\ref{th:UCBp-MPA-0} implies that
\begin{equation}
\label{eq:ucb.example}
\E r_n \leq \frac{1}{\a-1} \left( \frac{4 n}{6+K} - 1 \right)^{2(1-\a)} + \frac{K-2}{\a-1} \left( \frac{n}{6+K} - 1 \right)^{2(1-\a)}
\end{equation}
for all $n$ sufficiently large, {e.g.},
\begin{equation}
\label{eq:ass.on.n2}
n \geq \max \,\, \left\{ (K+2)(6+K), \ (4\a+1)\left( \frac{6+K}{4\Delta^2}\right)\ln n \right\}~.
\end{equation}
Now, the upper bound on $\E r_n$ given in Proposition~\ref{prop:unif}
for the uniform allocation associated with the recommendation provided by the empirical best arm is larger than
\[
\Delta e^{-\Delta^2 \lfloor n/K \rfloor}~, \qquad \mbox{ for all } n\geq K.
\]
Thus for $n$ moderately large, e.g., such that $n \geq K$ and
\begin{equation}\label{eq:ass.on.n}
\lfloor n/K \rfloor \leq (4\a+1) \left( \frac{6+K}{4\Delta^2} \right)\frac{\ln n}{K}~,
\end{equation}
the bound for the uniform allocation is at least
\[
\Delta \exp \left( -\Delta^2 (4\a+1) \left( \frac{6+K}{4\Delta^2} \right) \, \frac{\ln n}{K} \right)
= \Delta n^{ - (4\a+1) (6+K)/4K }~,
\]
which may be much worse than the upper bound (\ref{eq:ucb.example}) for the UCB$(\a)$ strategy
whenever $K$ is large, as can be seen by comparing the exponents $-2(\a-1)$ versus $-(4\a+1) (6+K)/4K$.

The reason is that the uniform allocation strategy only samples $\lfloor n/K\rfloor$ each arm,
whereas the UCB strategy focuses rapidly its exploration on the better arms.

%

\bibliographystyle{plainnat}
\bibliography{Biblio}

\end{document}